\documentclass[12pt]{article}
\usepackage[utf8]{inputenc}
\usepackage[a4paper,top=2.5cm,bottom=3cm,left=2.5cm,right=2.5cm,marginparwidth=2cm]{geometry}
\usepackage[hidelinks]{hyperref}
\usepackage{amsmath, amssymb, amsthm, bbm, graphicx, subcaption, tikz, subcaption, siunitx, enumerate, cleveref}
\usepackage[numbers,sort&compress]{natbib}

\renewcommand{\i}{\mathrm{i}} 
\newcommand{\Z}{\mathbb{Z}}
\newcommand{\R}{\mathbb{R}}

\newtheorem{thm}{Theorem}
\newtheorem{prop}[thm]{Proposition}

\theoremstyle{definition}

\numberwithin{equation}{section}
\numberwithin{thm}{section}

\title{Limit cycles govern wave energy accumulation in space–time-varying materials}
\author{Bryn Davies\thanks{Mathematics Institute, Zeeman Building, University of Warwick, Coventry CV4 7AL, UK} \and Erik Orvehed Hiltunen\thanks{$^1$Department of Mathematics, University of Oslo, Moltke Moes vei 35, 0851 Oslo, Norway}}
\date{}

\usepackage[textsize=scriptsize]{todonotes}

\begin{document}
	\maketitle

	\begin{abstract}
        We study the amplification and concentration of wave energy in smooth, periodic space–time media through the dynamics of the characteristic curves of the associated wave equation. In an impedance-matched setting, we show that stable limit cycles attract neighbouring characteristics, focusing the wave energy density. The exponential attraction rate is estimated analytically from the system's Lyapunov exponent. We further show that the attraction of characteristics and energy accumulation persist under a range of nonperiodic perturbations. Finally, we present a Floquet–Fourier numerical method for computing stable and unstable limit cycles and their attraction rates. Throughout, we exemplify the results using the example of modulation given by a superposition of travelling waves, revealing saddle--node bifurcations where limit cycles may appear depending on the relative wave speeds.  Full-wave simulations confirm that the characteristic mechanism continues to describe energy accumulation beyond the impedance-matched case.
	\end{abstract}
	
	\textbf{Keywords:} limit cycles; space–time materials; moving gratings; black holes; energy blowup; parametric amplification; characteristic curves; Floquet theory; Lyapunov exponents
		
	\section{Introduction}
    The study of wave propagation in time-modulated materials has seen a wealth of recent interest. Giving engineers the ability to manipulate material properties in time as well as space significantly alters the design space and leads to exotic phenomena, such as non-reciprocal transmission \cite{oliner1961wave, simon1960action}, Fresnel drag \cite{huidobro2019fresnel}, exceptional points \cite{nikzamir2022achieve}, and topological localisation \cite{rechtsman2013photonic}. The time-varying nature of these systems breaks the time-translational symmetry, and hence the energy conservation, of the system. Because of this, time-periodic systems give rise to frequency conversion and parametric amplification \cite{cullen1958travelling,holberg1966parametric, koutserimpas2018parametric}. Nevertheless, the energy is supplied through changing the material parameters and can be tracked. As result, such systems  satisfy generalized energy conservation relations which describe the energy sources in terms of the time-varying refractive index \cite{hiltunen2026optical, hirosawa2009generalised, zhang2024conservation}.

    The broken energy conservation of space–time-periodic materials also leads to the existence of stable and unstable limit cycles \cite{horsley2023quantum, lurie2009mathematical}. These are characteristic curves of the associated wave equation which are periodic with respect to the lattice and, additionally, serve as attractors of neighbouring characteristics. As a result, waves in such systems are spontaneously focused, leading to accumulation of the energy density and a blowup of the total energy. In particular, this  energy accumulation may appear in impedance-matched systems: in such systems, the overall field remains bounded and the energy blowup is due to wave focusing rather than blowup of the wave field itself. The wave focusing due to stable limit cycles is a direct analogy of self-focusing seen in nonlinear wave systems \cite{moll2003self}, where the index of refraction depends on the wave field propagating through the material and effectively creates a time-varying medium. In space–time metamaterials, the time-varying index is instead created through external forcing, resulting in a linear problem with qualitatively similar response.

    Periodic solutions and stable limit cycles have been thoroughly studied in canonical examples. For the classical example of travelling-wave modulation, where parameters are modulated with a moving profile of the form $f(x-vt)$, the accumulation of characteristic curves can be studied by exploiting a co-moving coordinate system \cite{horsley2023quantum, horsley2024traveling}. There have also been studies of piecewise constant ``rectangular''/``checkerboard'' space–time materials, where the refractive index instantaneously jumps between specific values \cite{lurie2025mechanism, lurie2009mathematical, lurie2006wave, milton2017field}. In this case, the associated characteristic curves are piecewise linear. This enables explicit characterisation of the limit cycles and energy accumulation by integrating the equations on each piecewise constant domain.
    
    Since the piecewise constant case is well-understood, we instead focus on a smoothly varying refractive index in this work. As we shall see, such systems naturally lend themselves to analytic perturbation theory and computational Floquet–Fourier methods. We will study the attraction of characteristics to limit cycles, and use this to predict the energy accumulation phenomena. Our methods can handle a general class of modulation. To demonstrate the key ideas, we will focus on an example where the modulation is the superposition of two travelling waves at different speeds, i.e. the material parameters have the form $f(x-mvt)+g(x-nvt)$ for integers $m\neq n$.

    After introducing the model in \Cref{sec:first}, we begin by recalling the main results in the case of a travelling-wave modulation in \Cref{sec:impedance}. In this case, the characteristic equation can be reduced to an autonomous ODE by changing to a moving coordinate frame. As a result, the exponential rate $r$ of convergence of the characteristics can be explicitly computed, and we use this to exemplify the energy accumulation due to limit cycles. These results are subsequently generalised to general space–time-varying profiles of the refractive index in \Cref{sec:main}. Specifically, we use analytic perturbation theory to determine the rate $r$ of attraction of stable limit cycles. This result is used to prove energy accumulation using scaling arguments following \cite{lurie2009mathematical}. We exemplify the results in the case of a double travelling-wave modulation. This is a simple system that still undergoes multiple saddle--node-type bifurcations where stable limit cycles appear and disappear depending on the parameters of the problem.
    
    Having characterised the energy accumulation as a consequence of the convergent characteristics, we turn to the question of \emph{stability} with respect to perturbations. We show that a small, periodic perturbation preserves the stable limit cycles. Similarly, while a small non-periodic perturbation may break the periodicity of the limit cycles, we show that the energy accumulation property is preserved. In the case of random noise, there is a competition of mechanisms: the underlying limit cycles serves to focus the characteristics, while the randomness introduces a diffusive term for the probability distribution of the characteristics. Finally, in \Cref{sec:floquet}, we introduce a Floquet–Fourier computational method for computing stable and unstable limit cycles: using Fourier expansions, we rephrase the characteristic ODE as a nonlinear system of equations which is readily solved using iterative methods.  
	
	\section{Problem setting}\label{sec:first}
	We consider electromagnetic waves in a space–time-varying grating following the model appearing in \cite{horsley2024traveling}
    \begin{equation} \label{eq:wave}
		\frac{\partial}{\partial x} \left( \frac{1}{\mu} \frac{\partial u}{\partial x} \right) - \frac{1}{c_0^2}\frac{\partial}{\partial t} \left( \epsilon\frac{\partial u}{\partial t} \right)=0,
	\end{equation}
	where $c_0$ is the background speed of light. Here, both the permeability $\mu=\mu(x,t)$ and the permittivity $\epsilon=\epsilon(x,t)$ are allowed to vary as functions of space and time. In this framework, the electric and magnetic fields are defined by $u$ as  $\mathbf{E} = -(\partial_t u) \mathbf{e}_y$ and $\mathbf{B} = (\partial_x  u )\mathbf{e}_z$, respectively \cite{horsley2024traveling}. We are interested in parameters that are periodic in space–time, in the sense that there exist $L$ and $T$ such that
	\begin{align} \label{periodicity}
		\mu(x+L,t)=\mu(x,t) \quad\text{and}\quad \mu(x,t+T)=\mu(x,t),
	\end{align}
    for all $x$ and $t$, along with the same periodicity for $\epsilon$. This is the key assumption to obtain characteristic curves which are a periodic limit cycles. Throughout, we assume that $\mu(x,t)>0$ and  $\epsilon(x,t)>0$ for all $(x,t)$. As we will see, the emergence of limit cycles will cause energy accumulation and blowup. In the model under consideration, the energy density $\rho(x,t)$ is given by (see, for example, \cite{koutserimpas2020electromagnetic})
	\begin{equation}\label{eq:E}
		\rho(x,t) = \frac{\epsilon(x,t)}{2c_0^2} \left(\frac{\partial u}{\partial t}\right)^2 + \frac{1}{2\mu(x,t)} \left(\frac{\partial u}{\partial x}\right)^2,
	\end{equation}
	and the total energy of the solution $u$ at time $t$ as 
	\begin{equation}
		E(t) = \int_\R \rho(x,t) \mathrm{d}x.
	\end{equation}
    
	\section{Impedance-matched systems}\label{sec:impedance}
	A useful simplification is to assume that we have impedance matching, such that $\epsilon=\mu=n$, where $n=n(x,t)$ is the refractive index. Under this assumption, \eqref{eq:wave} can be factorised into two first-order equations
	\begin{equation} \label{matched_odes}
		\frac{c_0}{n(x,t)}\frac{\partial u}{\partial x} \pm \frac{\partial u}{\partial t}=0,
	\end{equation}
	where the plus and minus signs correspond to waves propagating to the right and to the left, respectively. In particular, this assumption decouples waves travelling in each direction, and is a common assumption made when studying energy accumulation in time-varying systems \cite{lurie2006wave,horsley2024traveling,pendry2021gain}. 	
	For impedance-matched systems, the energy density simplifies to 
	\begin{equation}\label{eq:E2}
		\rho(x,t) = \frac{1}{n(x,t)} \left(\frac{\partial u}{\partial x}\right)^2.
	\end{equation}
    Throughout, we consider rightward-propagating waves
	\begin{equation}\label{eq:rightward}
		\frac{c_0}{n(x,t)}\frac{\partial u}{\partial x} + \frac{\partial u}{\partial t}=0,
	\end{equation}
	and study properties of the characteristic curves $x(t)$, which are solutions to the (nonlinear) equation
	\begin{equation}
		\frac{\mathrm{d} x}{\mathrm{d}t} =  \frac{c_0}{n(x,t)}.
	\end{equation}
	
	\subsection{Travelling-wave modulation}\label{sec:travelling}
	We begin by illustrating the main ideas with the simplest and most widely studied case where the modulation is a fixed profile $f$ moving at a constant speed $v$:
	\begin{equation}
		\frac{c_0}{n(x,t)}=f(x-vt),
	\end{equation}
	for some $L$-periodic function $f(y)$. This case has previously been considered in a number of works, but we include it here as a simple, explicitly solvable model which illustrates the main features of energy accumulation due to stable limit cycles. The most interesting case is where the waves are moving in the same direction as the modulation, i.e. when $v>0$. We can make a change of variables
	\begin{equation}
		y=x-vt,
	\end{equation}
    so that the variable $y$ is the spatial coordinate with respect to a reference frame that moves along with the modulation. This leads to
	\begin{equation}\label{eq:char_travelling}
		\frac{\mathrm{d}y}{\mathrm{d}t}=f(y)-v.
	\end{equation}
	Solutions of \eqref{eq:char_travelling} are straight lines $y(t) = y_0$ for all $y_0$ such that $f(y_0)=v$, and such equilibria are stable whenever $f'(y_0) < 0$. When $f(y)\neq v$, eq. \eqref{eq:char_travelling} has general solution satisfying
	\begin{equation} \label{charac_integral}
		\int\frac{1}{f(y)-v}\mathrm{d}y=t+C.
	\end{equation}
    This integral can be handled in some simple cases and for the remainder of this introductory section, we suppose that $f(y)=2+\cos(y)$. At the points $v=1$ and $v=3$, the problem has saddle--node bifurcations where (modulo $L$) one stable and one unstable limit cycle emerges for $v\in (1,3)$. In this region, we have \cite[eq. \textbf{2.551} 3]{gradshteyn2014table}
	\begin{equation}
		\int\frac{1}{2-v + \cos(y)}\mathrm{d}y = \frac{1}{\sqrt{1-(2-v)^2}}\log \frac{(2-v)\tan\frac{y+\pi/2}{2} + 1 -\sqrt{1-(2-v)^2}}{(2-v)\tan\frac{y+\pi/2}{2} + 1 +\sqrt{1-(2-v)^2}}.
	\end{equation}
	This gives solutions 	
	\begin{equation}
		y(t)= -\frac{\pi}{2}+2\arctan\left(\frac{r+1 + (r -1)\exp(-rt+C) }{(v-2)\bigl(1- \exp(-rt+C)\bigr)}\right),
	\end{equation}
	where $r = \sqrt{1-(2-v)^2}$. As $t\to \infty$, we have  
	\begin{equation}
		\lim_{t\to \infty} y(t) =   -\frac{\pi}{2}+2\arctan\left(\frac{1 +r}{v-2}\right) = \arccos\left(v-2\right),
	\end{equation}
	in other words, the characteristic $y(t)$ approaches the stable limit cycle $y_0$ such that $f(y_0) = v$. Moreover, $y(t)$ converges to $y_0$ at an exponential rate given by $r$:
	\begin{equation}\label{eq:r_travelling}
		\bigl|y(t) - y_0\bigr| = O\bigl(\exp(-rt)\bigr).
	\end{equation}
    We highlight that $r$ can alternatively be written $r=f'(y_0)$. The concentration of characteristics is shown in \Cref{fig:travelling}, and we observe the predicted exponential rate $r$ of convergence in \Cref{fig:error_travelling}.
    
    Next, we will illustrate how the attractive limit cycle gives rise to energy accumulation. For initial conditions $y(0) = y_0 + \epsilon$ close to $y_0$, the characteristics are given by 
    \begin{equation}
    y(t) = y_0 + \epsilon\exp(-rt) + O\bigl(\epsilon^2\exp(-2rt) \bigr).
    \end{equation}
    For the purpose of illustration, we neglect the error term and define $p(x,t) = y_0 + (x-vt-y_0)\exp(rt)$. In this case, the solution to equation \eqref{eq:rightward} satisfying $u(x,0) = u_0(x)$ is given by 
    \begin{equation}
        u(x,t) = u_0\bigl(p(x,t)\bigr),
    \end{equation}
    and it directly follows that the energy density $\rho$ and total energy $E$ satisfy the scaling
    \begin{equation}\label{eq:accumuation_travelling}
    \rho(x,t) = \exp(2rt)\rho\bigl(p(x,t),0\bigr), \quad E(t) = \exp(rt)E(0).
    \end{equation}
    In other words, the initial wave is focused so that the energy density accumulates along the limit cycle and the total energy diverges. This is illustrated in \Cref{fig:energy_travelling}, where we plot the energy density when the travelling-wave-modulated system from \Cref{fig:travelling} is initialised by a Gaussian pulse $u(x,0) = \exp(-x^2/0.4L)$. As shown in \eqref{eq:accumuation_travelling}, the energy accumulates along the limit cycles.
    
	\begin{figure}[tbh]
		\begin{subfigure}[t]{0.33\linewidth}
			\centering
			\includegraphics[width=\linewidth]{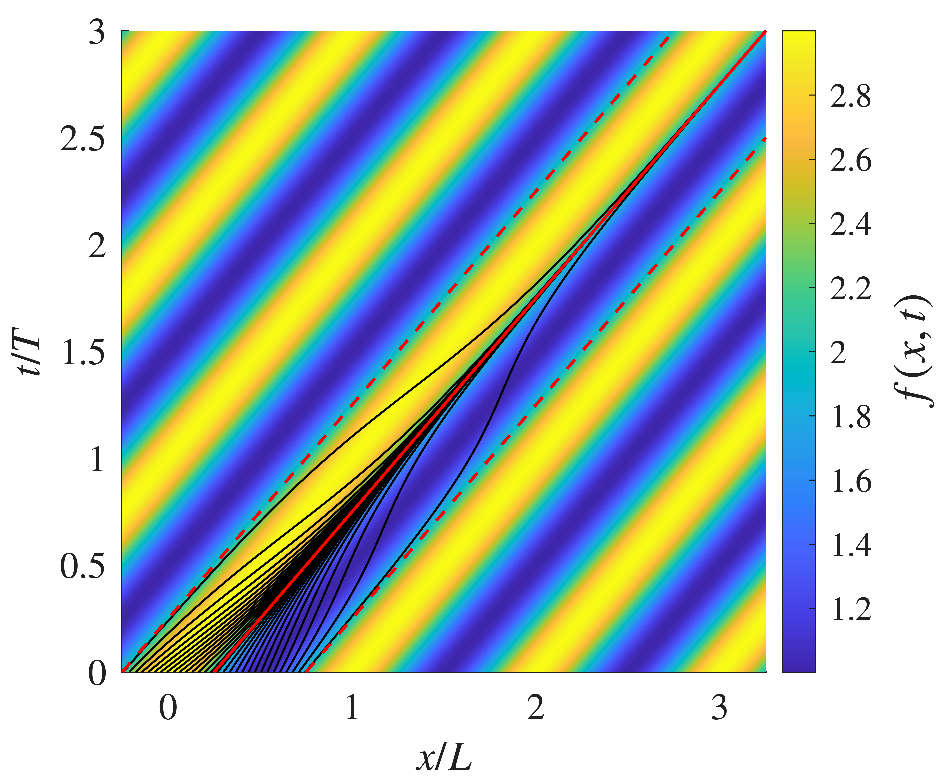}
			\caption{Travelling-wave modulation.}\label{fig:travelling}
		\end{subfigure}\hfill
		\begin{subfigure}[t]{0.33\linewidth}
			\centering
			\includegraphics[width=\linewidth]{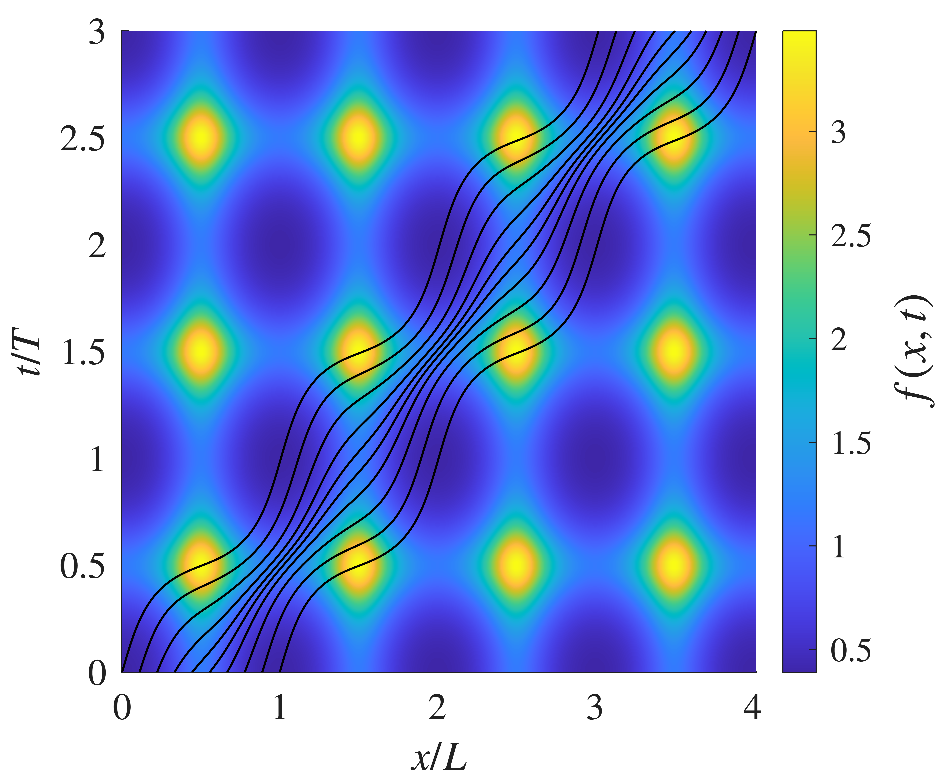}
			\caption{Separable modulation} \label{fig:separable}
		\end{subfigure}\hfill
		\begin{subfigure}[t]{0.33\linewidth}
			\centering
			\includegraphics[width=\linewidth]{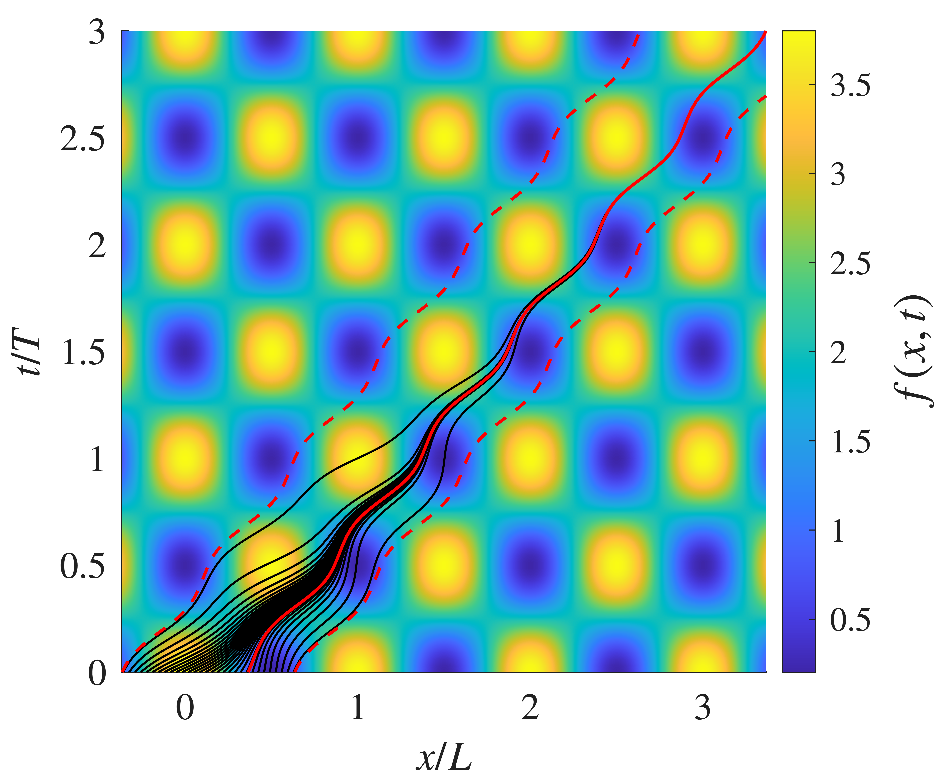}
			\caption{Standing-wave modulation} \label{fig:standing}
		\end{subfigure}
        \caption{When the refractive index is given by a travelling wave (a) or standing wave (c), the characteristics diverge from the unstable limit cycles (dashed lines) and tend to the stable limit cycle (bold line) with exponential rate $r$. In the case of separable modulation, shown in (b) there cannot be attractive limit cycles, although characteristics may still be periodic with respect to the lattice.}
	\end{figure}

	\begin{figure}[tbh]
		\begin{minipage}[t]{0.48\linewidth}
			\centering
			\includegraphics[width=\linewidth]{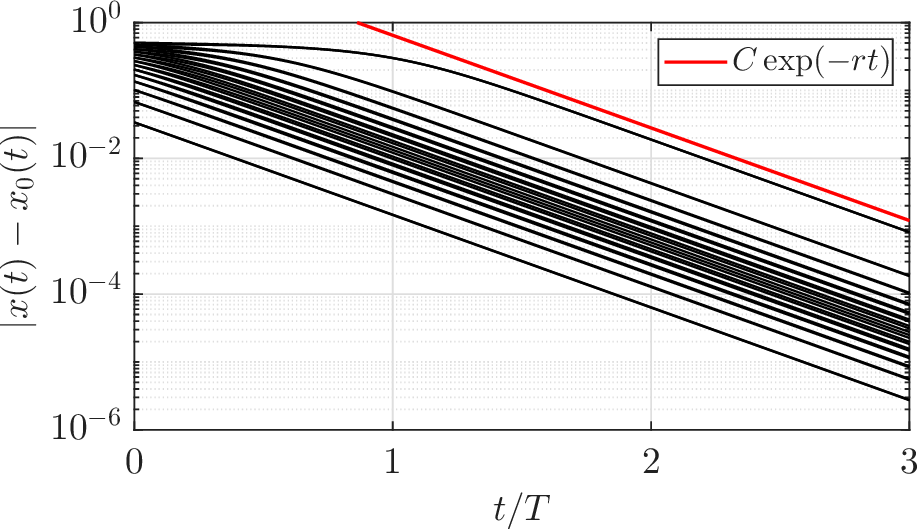}
			\caption{We plot the convergence of the characteristics from \Cref{fig:travelling} to the corresponding limit cycle and observe the predicted exponential rate $r$ of convergence.}\label{fig:error_travelling}
		\end{minipage}\hfill
		\begin{minipage}[t]{0.48\linewidth}
			\centering
			\includegraphics[width=\linewidth]{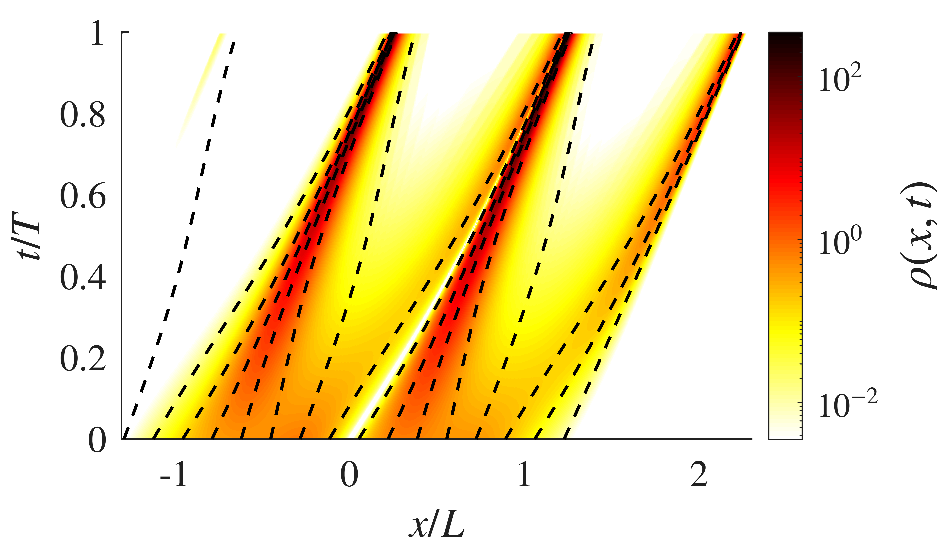}
			\caption{When the system from \Cref{fig:travelling} is initialised by a pulse, the convergent characteristics cause energy accumulation along the limit cycles.} \label{fig:energy_travelling}
		\end{minipage}
		
	\end{figure}
	
	\subsection{Separable modulation}
	If $n(x,t)$ is separable, in the sense that it can be written as the product of independent functions of space and time,
	\begin{equation}
		\frac{c_0}{n(x,t)}=f(x)g(t),
	\end{equation}
	there cannot be attractive characteristics of the wave equation. To see why, we consider the characteristic equation 
	\begin{equation}
		\frac{\mathrm{d} x}{\mathrm{d}t} = f(x)g(t).
	\end{equation}
	This can be solved by introducing the antiderivative $G(t) = \int_0^t g(\tau)\mathrm{d}\tau$. Since $g(t)>0$, we know that $G(t)$ is strictly increasing, so that $t = G^{-1}(\tau)$ is a well-defined change of coordinates that transforms the characteristic equation into 
	\begin{equation} \label{characs:separable}
		\frac{\mathrm{d} x}{\mathrm{d}\tau} = f(x).
	\end{equation}
	We now have an autonomous ODE with $f(x)\neq 0$, and there cannot be attractive limit cycles. An example of this is shown in Figure~\ref{fig:separable}, where we show the characteristics for the case where $f(x)=3.5/(2+\cos(2\pi x))$ and $g(t)=1/(2+\cos(2\pi t))$. Even though the characteristics are non-attractive, we highlight that they could be periodic with respect to the lattice $L\Z\times T\Z$.

    We highlight that a standing-wave modulation of the form
    \begin{equation}
		\frac{c_0}{n(x,t)}=C + f(x)g(t),
	\end{equation}
    may still support stable limit cycles, as studied in \cite{horsley2024traveling}. In \Cref{fig:standing}, we consider the case
    \begin{equation}\label{eq:standing}
		\frac{c_0}{n(x,t)}= 2+1.8\cos\left(\frac{2\pi x}{L}\right)\cos\left(\frac{2\pi t}{T}\right) 
    \end{equation}
    and observe stable and unstable limit cycles as in \Cref{fig:travelling}.

	\section{Limit cycle characteristics}\label{sec:main}
	We now consider a general space–time-varying refractive index $n(x,t)$, and aim to study the rate of attraction, and associated energy accumulation, of characteristic curves. The goal is to formulate a general framework for the behaviour illustrated in for travelling-wave modulations in \Cref{sec:travelling}. We define the instantaneous wave speed $f(x,t)$ as 
	\begin{equation}
		f(x,t) = \frac{c_0}{n(x,t)}.
	\end{equation}
	The characteristic curve $x(t)$ through $(a,0)$ is given by the solution to the (nonlinear) initial-value problem
	\begin{equation}\label{eq:char}
		\begin{cases}
			\displaystyle \frac{\mathrm{d} x}{\mathrm{d}t} = f(x,t), \\
			x(0) = a.
		\end{cases}
	\end{equation}
	Note that $f(x,t)$ is periodic in both $x$ and $t$; nevertheless, the solutions are not required to be periodic. Following \cite{lurie2009mathematical}, a solution $x(t)$ satisfying $x(t+qT) = x(t) + pL,$ for some integers $p,q\in \Z$ will be called a $(p,q)$-periodic \emph{limit cycle}. In other words, a  limit cycle is a characteristic curve which repeats periodically across  $p$ periods in $x$ and $q$ periods in $t$. 
	
	The solution to \eqref{eq:char} can be written implicitly as
	\begin{equation}
		x(t) = a+\int_0^t f(x(\tau),\tau)\mathrm{d}\tau.
	\end{equation}
	The condition for a limit cycle then reads
	\begin{equation}\label{eq:average}
		\frac{1}{qT}\int_0^{qT}f(x(t),t)\mathrm{d} t = \frac{pL}{qT}.
	\end{equation}
	In other words, the average instantaneous wave speed along the limit cycle must equal a rational multiple of the effective grating velocity $v= L/T$. Equivalently, we can define a moving coordinate frame $y(t) = x(t) - t\frac{pL}{qT}$, so that $y$ is a periodic solution to the problem
	\begin{equation}\label{eq:char_y}
		\begin{cases}
			\displaystyle \frac{\mathrm{d} y}{\mathrm{d}t} = f\bigl(y+t\tfrac{pL}{qT},t\bigr) - \frac{pL}{qT}, \\
			y(0) = a.
		\end{cases}
	\end{equation}
	Although most frequently seen with periodicity $(p,q)=(1,1)$, we  emphasise that limit cycles can appear with general periodicities $(p,q)$, as shown in \Cref{fig:doublewave}. Here, we consider a sum of two travelling wave modulations:
	\begin{equation}\label{eq:doublewave}
		f(x,t) = 2 + \cos\left(\frac{2\pi}{L}(x-vt)\right) + \frac{1}{2} \cos\left(\frac{2\pi}{L}(x-2vt)\right).
	\end{equation}
	Depending on the speed $v$ of the travelling waves, limit cycles may appear at different periodicities. Here, the limit cycles are computed using a Floquet–Fourier method (outlined in \Cref{sec:floquet}), while the remaining characteristic curves are computed by a direct numerical integration of \eqref{eq:char}. At the bifurcation point, the unstable (dashed lines) and stable (bold lines) limit cycles coincide in a one-sided unstable limit cycle in a saddle--node-type bifurcation. This transition is further highlighted in \Cref{fig:zoom}, where the projection of the limit cycles are plotted over the torus $[0,L]\times[0,T]$.
	
	\begin{figure}[tbh]
		\begin{subfigure}[t]{0.33\linewidth}
			\centering
			\includegraphics[width=\linewidth]{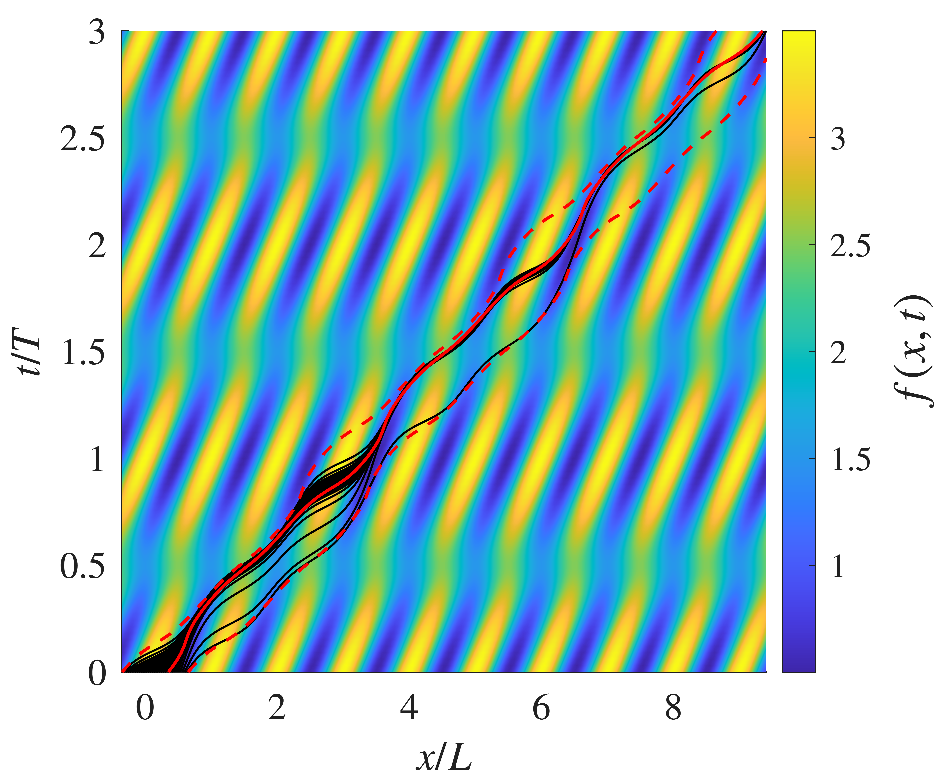}
			\caption{$v=0.5$}
		\end{subfigure}\hfill
		\begin{subfigure}[t]{0.33\linewidth}
			\centering
			\includegraphics[width=\linewidth]{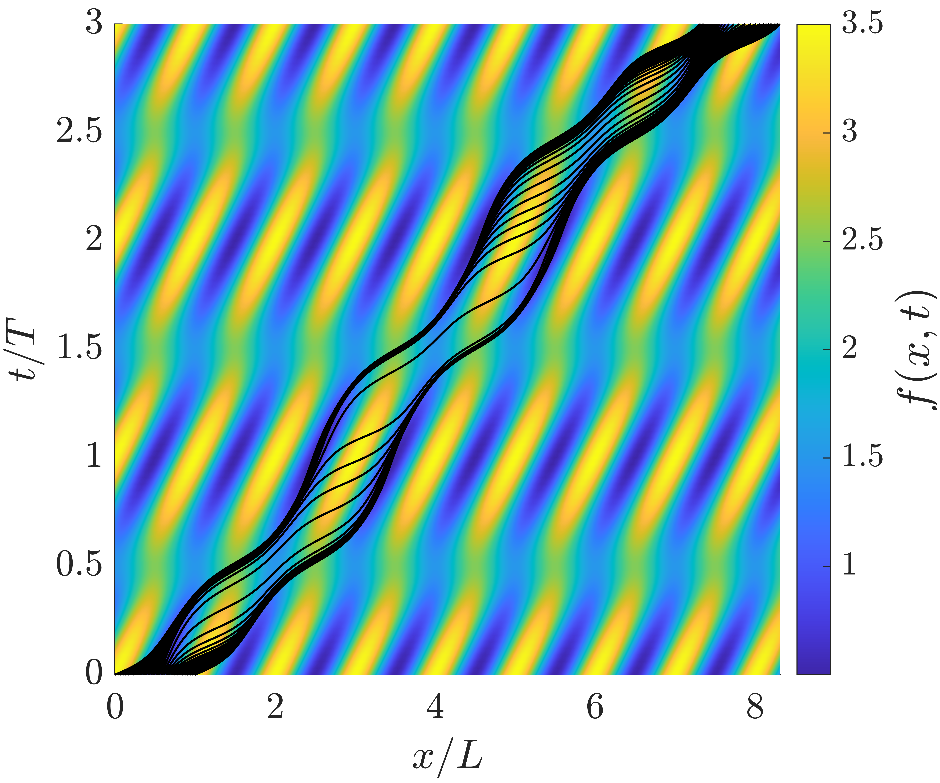}
			\caption{$v=0.54$}
		\end{subfigure}\hfill
		\begin{subfigure}[t]{0.33\linewidth}
			\centering
			\includegraphics[width=\linewidth]{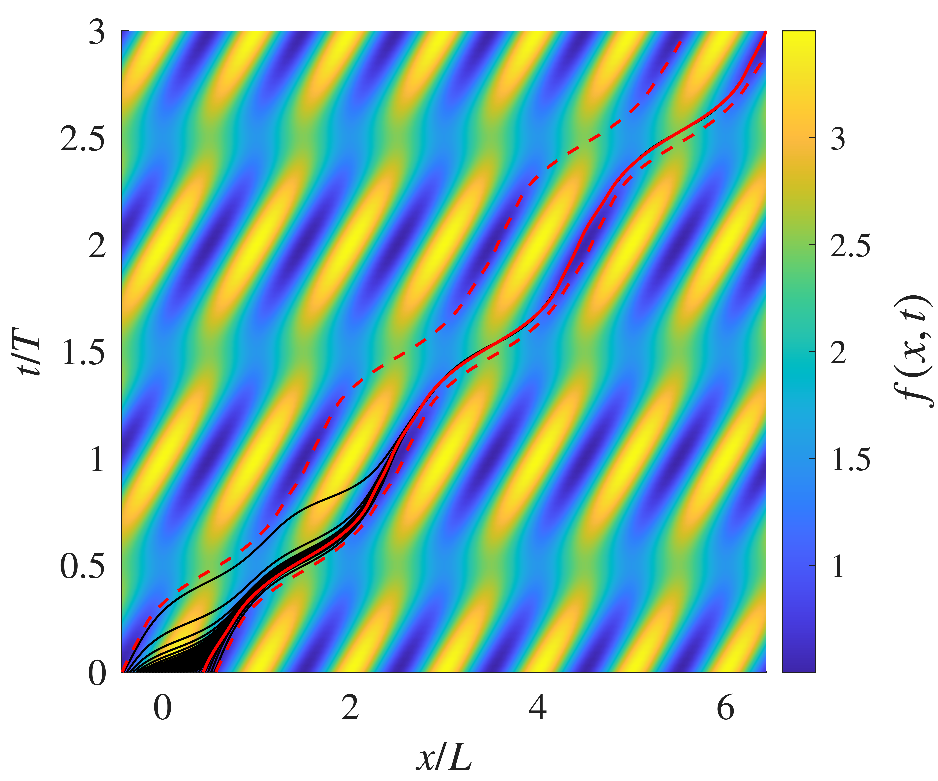}
			\caption{$v=0.58$}
		\end{subfigure}
		\caption{Transition between limit cycles with a sum of two travelling wave modulations $f(x,t) =  2 + \cos(2\pi/L(x-vt)) + 0.5\cos(2\pi/L(x-2vt))$. The limit cycles in the left panel (a) satisfy $p=3, q=1$ while those in the right panel (c) satisfy $p=2, q=1$. In the intermediate panel, no limit cycles appear.}\label{fig:doublewave}
		
		\vspace{0.3cm}
		
		\begin{subfigure}[t]{0.33\linewidth}
			\centering
			\includegraphics[width=\linewidth]{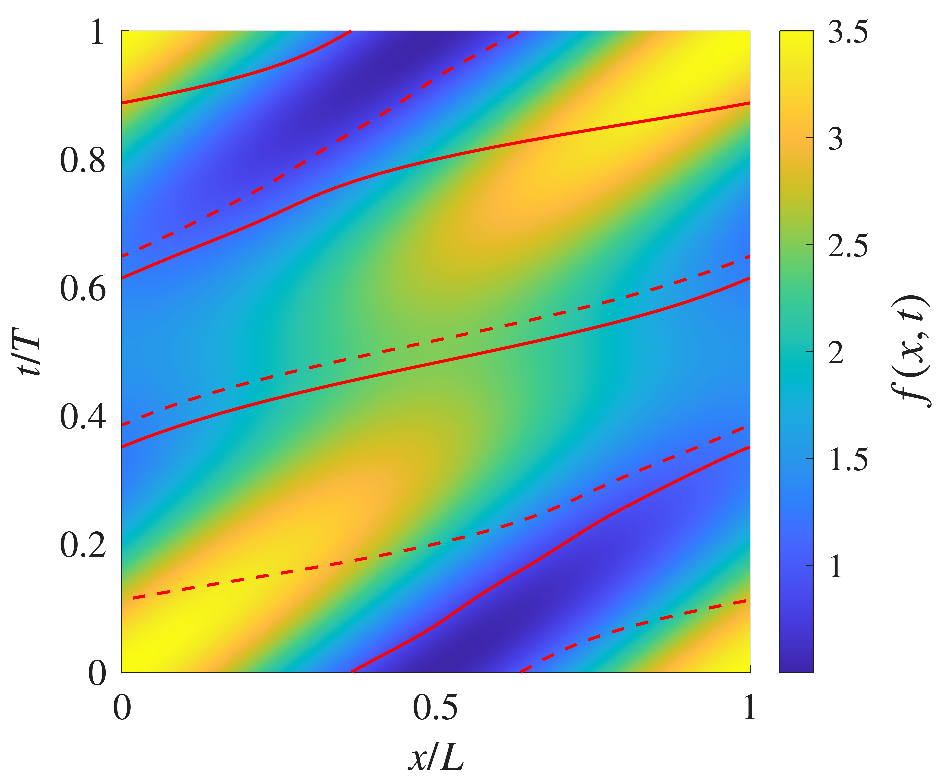}
			\caption{$v=0.5$}
		\end{subfigure}\hfill
		\begin{subfigure}[t]{0.33\linewidth}
			\centering
			\includegraphics[width=\linewidth]{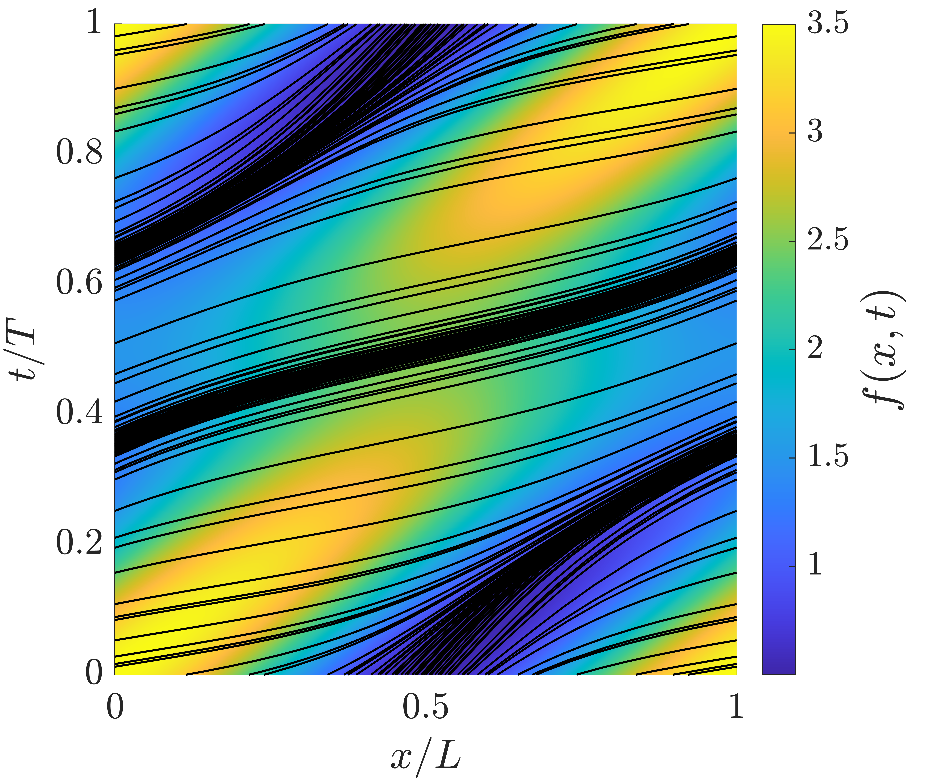}
			\caption{$v=0.54$}
		\end{subfigure}\hfill
		\begin{subfigure}[t]{0.33\linewidth}
			\centering
			\includegraphics[width=\linewidth]{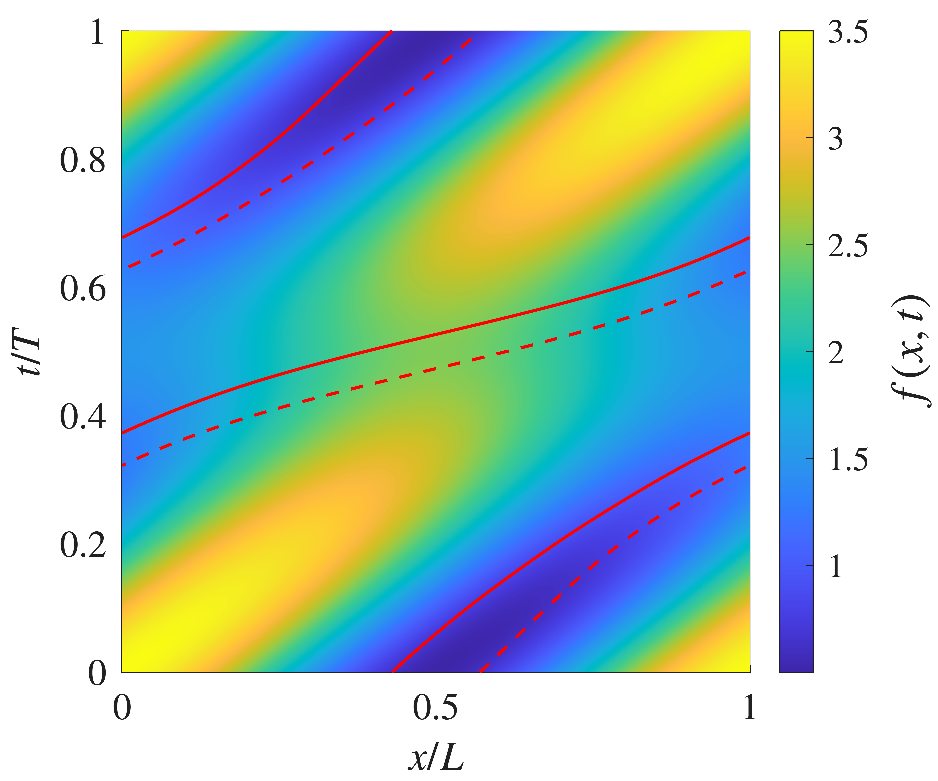}
			\caption{$v=0.58$}
		\end{subfigure}
		\caption{We project the limit cycles from \Cref{fig:doublewave} to the torus $[0,L]\times[0,T]$: stable limit cycles (bold lines) and unstable limit cycles (dashed). In (b), where there are no limit cycles, the curve does not repeat and a single characteristic curve instead densely fills the torus.}\label{fig:zoom}
	\end{figure}
	
	\subsection{Attractiveness of limit cycles}\label{sec:attraction}
	Using analytic perturbation theory, we can estimate the rate at which neighbouring characteristics approach an attractive limit cycle and generalise \eqref{eq:r_travelling}. We consider the problem \eqref{eq:char_y} and assume we have a periodic solution $y_0(t)$ with initial conditions $y_0(0) = a_0$.
	\begin{prop}\label{prop:asymptotic}
		Assume $x_0(t)$ is a $(p,q)$-periodic limit cycle of an analytic, $(L,T)$-periodic function $f(x,t)$ satisfying $x_0(0) = a_0$. Assume that 
		\begin{equation}\label{eq:stability}
			-\frac{1}{qT}\int_0^{qT} f_x(x_0(t),t)\mathrm{d} t := r > 0.\end{equation}
		Given a constant $a_1$ and for $0<\epsilon\ll1$, let $a_\epsilon = a_0 + a_1\epsilon$ and let $x_\epsilon(t)$ be the solution of \eqref{eq:char} satisfying $x_\epsilon(0) = a_\epsilon$. Then, for small enough $\epsilon$, this solution satisfies
		\begin{equation}\label{eq:attractive}
			\bigl|x_\epsilon(t) - x_0(t)\bigr| = O\bigl(\epsilon\exp(-rt)\bigr),
		\end{equation}
		as $t\to \infty$.
	\end{prop}
	\begin{proof}
		For small $\epsilon$, we let $y_\epsilon(t) = x_\epsilon(t)- \frac{pL}{qT}t$ be the solution of \eqref{eq:char_y} with initial condition  $a_\epsilon = a_0+\epsilon a_1$. From ODE-theory, we know that $y$ depends analytically on $\epsilon$, so we may write
		\begin{equation}
			y_\epsilon(t) = \sum_{n=0}^\infty \epsilon^n y_n(t).
		\end{equation}
		It is straight-forward to show that $y_1$ solves the problem
		\begin{equation}\label{eq:linearised_1}
			\begin{cases}\displaystyle \frac{\mathrm{d} y_1}{\mathrm{d} t} = f_x(x_0(t),t)y_1, \\ y_1(0) = a_1,
			\end{cases}
		\end{equation}
		and for $n>1$ the coefficients $y_n$ solve the problem 
		\begin{equation}\label{eq:linearised_n}
			\begin{cases}\displaystyle \frac{\mathrm{d} y_n}{\mathrm{d} t} = f_x(x_0(t),t)y_n + g_n(t), \\ y_n(0) = 0,
			\end{cases}
		\end{equation}
		for some function $g_n$ defined solely on $y_m$ for $m<n$. We note that an explicit formula for $g_n$ can be found through Fa\`a di Bruno's formula, although the result is increasingly complicated as $n$ increases. We define the antiderivative $F(t)$ of $f_x(x_0(t),t)$:
		\begin{equation}\label{eq:F}
			F(t) = \int_0^t f_x(x_0(\tau),\tau) \mathrm{d}\tau.
		\end{equation}
		Note that the integrand is periodic and $F$ scales like $F(t) = O(rt)$ as $t\to \infty$.
		
		For $n\geq 1$, we will show inductively that $y_n(t) = O\bigl(\exp(-rnt)\bigr)$ as $t\to \infty$. For $n=1$, this follows directly from \eqref{eq:linearised_1}. From the induction hypothesis and Fa\`a di Bruno's formula, it follows that $g_n(t) = O\bigl(\exp(-rnt)\bigr)$. This gives
		\begin{equation}
			y_n(t) = \exp\bigl(F(t)\bigr)\int_0^t g_n(\tau)\exp\bigl(-F(\tau)\bigr) \mathrm{d} \tau,
		\end{equation}
		so that $\bigl|y_n(t)\bigr|= O\bigl(\exp(-rnt)\bigr)$ as $t\to \infty$. Hence, the solution $y_\epsilon$ satisfies 
		\begin{equation}
			\bigl|y_\epsilon(t) - y_0(t)\bigr| = O\bigl(\epsilon\exp(-rt)\bigr),
		\end{equation}
		which proves the claim.
	\end{proof}
	In \Cref{fig:error}, we plot the convergence of the characteristics of \Cref{fig:doublewave}(c), and observe the expected exponential rate of $r$ defined in \eqref{eq:stability}. Interestingly, there may be regions at which $f_x$ is positive, which results in local spreading of the characteristic curves. Still, if the average of $f_x$ is negative, the characteristics will converge to the limit cycle.
    
    Since the first-order term $y_1(t)$ is explicit, \eqref{eq:attractive} can be improved by the following estimate,
	\begin{equation}\label{eq:attractive2}
		x_\epsilon(t) = x_0(t)+\epsilon a_1\exp\bigl(F(t)\bigr) + O\bigl(\epsilon^2\exp(-2rt)\bigr),
	\end{equation}
	as $t\to \infty$, where $F(t)$ is given by \eqref{eq:F}. Moreover, as $\epsilon\to 0$ we have that 
	\begin{equation}
		x_\epsilon(qT) = x_0(qT) + \epsilon a_1 \exp(-rqT) + O(\epsilon^2),
	\end{equation}
	For small enough $\epsilon$, the characteristics approach the limit cycles at a rate of $\alpha:= \exp(-rqT)$ across one period of $qT$.
	
	\begin{figure}
		\begin{minipage}[t]{0.48\linewidth}
			\centering
			\includegraphics[width=\linewidth]{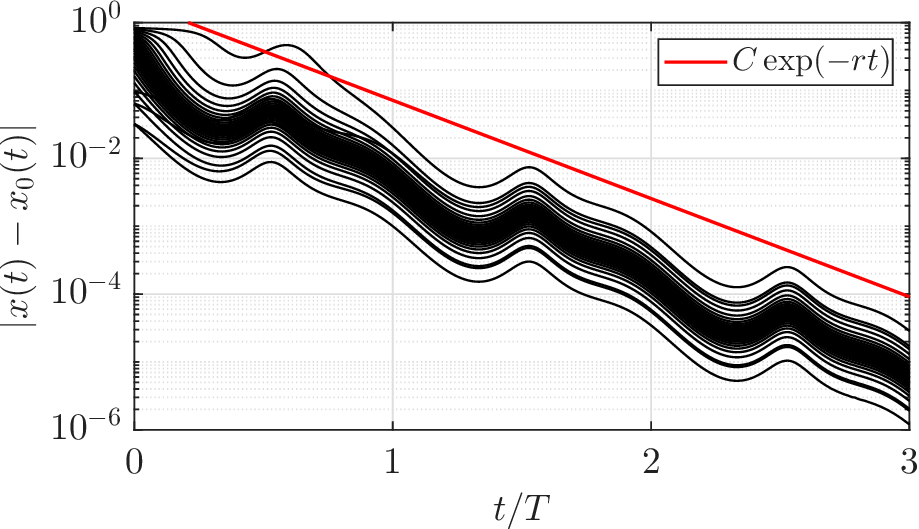}
			\caption{As $t\to \infty$, the nearby characteristics converge exponentially to an attractive limit cycle, with rate given by $r$ defined in \eqref{eq:stability}. Here, we consider the setting of \Cref{fig:doublewave}(c) and plot the difference between the characteristics $x(t)$ and the limit cycle $x_0(t)$. As seen, the characteristics converge to the limit cycle with the expected exponential rate.}\label{fig:error}
		\end{minipage}\hfill
		\begin{minipage}[t]{0.48\linewidth}
			\centering
			\includegraphics[width=\linewidth]{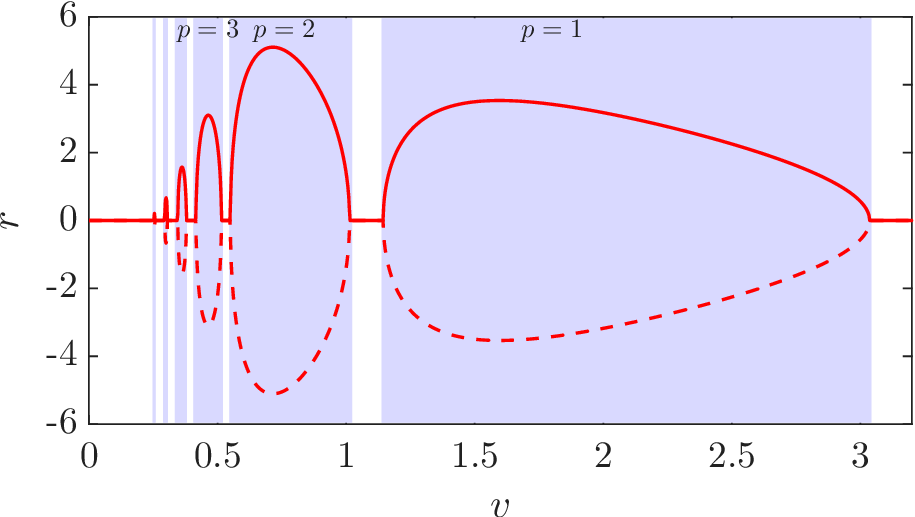}
			\caption{As we vary $v$, stable limit cycles appear and disappear for the double-travelling wave modulation. These are measured by the rate of attraction $r$, where $r > 0$ for stable limit cycles (bold line) and $r<0$ for unstable limit cycles (dashed line). Notably, as $v$ varies, the limit cycles appear at different periodicities $p$.}\label{fig:cinfty}
		\end{minipage}
		
	\end{figure}

	The proof of \Cref{prop:asymptotic} shows that any characteristic curve $x(t)$ (not necessarily periodic), will attract neighbouring characteristics at a variable rate with limiting value
	\begin{equation}\label{eq:rt}
		r_\infty = \lim_{t\to\infty}-\frac{1}{t}\int_0^t f_x(x(\tau),\tau) \mathrm{d}\tau.
	\end{equation}
	In \Cref{fig:cinfty}, we plot $r_\infty$ for the double-travelling wave modulation \eqref{eq:doublewave} over a range of values $v$. This allows us to distinguish parameter regions for which stable limit cycles appear, and further highlights the bifurcation shown in \Cref{fig:doublewave} and \Cref{fig:zoom}: the intermediate value $v=0.54$ (b) falls between the regions associated to $p=3$ (a) and $p=2$ (c).

	\subsection{Energy accumulation}
    A formal argument shows how the energy accumulation results from the convergent characteristics. Assume we have an attractive limit cycle of \eqref{eq:char} with Lyapunov coefficient $\alpha=\exp(-rqT)<1$ in units of the time-period $qT$. Across one period in $t$, we expect each derivative appearing in \eqref{eq:E} to scale like $1/\alpha$, while the support scales like $\alpha$, giving an overall scaling like
	\begin{equation}
		E(t+qT) \sim \frac{1}{\alpha}E(t).
	\end{equation}
	In other words, the scalings imply that the energy increases by a factor of $1/\alpha$ across each time period. In the next result, we make these formal arguments precise using the results of \Cref{sec:attraction}.
	\begin{prop}\label{prop:energy}
		Assume $x_0(t)$ is a $(p,q)$-periodic limit cycle of an analytic, $(L,T)$-periodic function $f(x,t)$ satisfying $x_0(0) = a_0$ such that \eqref{eq:stability} holds. At $t=0$, assume that $u(x,0) = u_0(x)$ is supported on $\Omega = a_0+[-\epsilon,\epsilon]$ for some $0<\epsilon\ll 1$. Then, for small enough $\epsilon$, 
		\begin{equation}
			E(qT) = \exp(rqT)E(0)+O\left(\epsilon\|u_0'\|_\infty^2\right).
		\end{equation}
	\end{prop}
	\begin{proof}
		We define the flow map $P(z,t)$ as the solution to the characteristic equation
		\begin{equation}
			\begin{cases}
				\displaystyle \frac{\partial P}{\partial t} = f(P(z,t),t), \\
				P(z,0) = z.
			\end{cases}
		\end{equation}
		Then $P_z(z,t)$ solves the equation
		\begin{equation}\label{eq:P}
			\begin{cases}
				\displaystyle \frac{\partial P_z}{\partial t} = f_x\bigl(P(z,t),t\bigr)P_z, \\
				P_z(z,0) = 1.
			\end{cases}
		\end{equation}
		We now let $p(z)=P(z,qT)$ be the flow of $z$ across one period $qT$, then $p$ is invertible from $\Omega$ onto its image. From \eqref{eq:P}, we have 
		\begin{equation}\label{eq:p'}
			p'(z) = \exp\left(\int_0^{qT}f_x\bigl(P(z,s),s\bigr)\mathrm{d} s\right).
		\end{equation}
		We apply the substitution $x =p(z)$ to the integral defining $E$,
		\begin{align}
			E(qT) &= \frac{1}{c_0}\int_{p(a_0-\epsilon)}^{p(a_0+\epsilon)} f(x,qT)\left(\frac{\partial }{\partial x}u(x,qT)\right)^2 \mathrm{d} x \\
			&= \frac{1}{c_0}\int_{a_0-\epsilon}^{a_0+\epsilon} f(p(z),0) \left(\frac{\partial}{\partial x}u_0\left(p^{-1}(x)\right)\right)^2 p'(z) \mathrm{d} z\\
			&= \frac{1}{c_0}\int_{a_0-\epsilon}^{a_0+\epsilon} f(p(z),0)\left(u_0'(z)\right)^2 \frac{\mathrm{d} z}{p'(z)}.\label{eq:int}
		\end{align}
		This result can be estimated by a first-order Taylor approximation as $\epsilon\to 0$. From \eqref{eq:p'} we have
		\begin{equation}
			p(z) = a_0+pL +O(\epsilon), \quad f(p(z),0) = f(z,0) + O(\epsilon), \quad p'(z) = \exp(-rqT) + O(\epsilon),
		\end{equation}
		valid uniformly for $z\in \Omega$. This gives in total
		\begin{equation}
			E(qT)= \exp(rqT)\left(\frac{1}{c_0}\int_{a_0-\epsilon}^{a_0+\epsilon} f(z,0)\left(u_0'(z)\right)^2\mathrm{d} z\right)+O\left(\epsilon\|u_0'\|_\infty^2\right),
		\end{equation}
		which proves the claim.
	\end{proof}
	
	\begin{figure}
		\begin{subfigure}[t]{0.48\linewidth}
			\centering
			\includegraphics[width=0.9\linewidth]{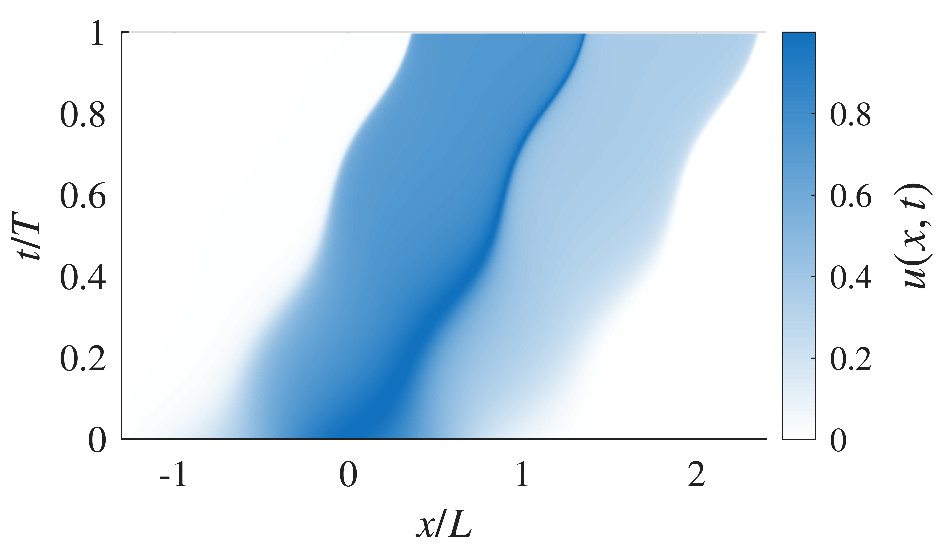}
			\caption{}
		\end{subfigure}\hfill
		\begin{subfigure}[t]{0.48\linewidth}
			\centering
			\includegraphics[width=0.9\linewidth]{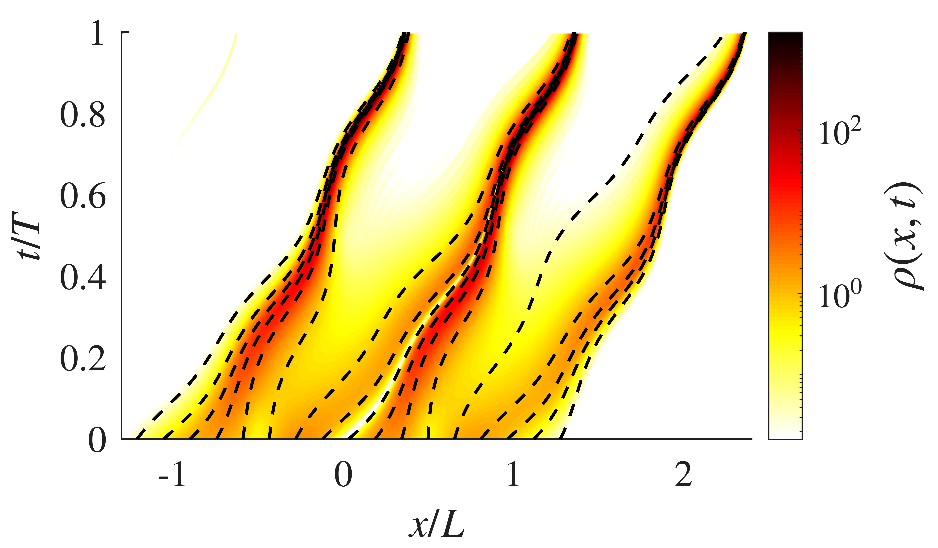}
			\caption{}
		\end{subfigure}
		\caption{The energy accumulation observed in \Cref{fig:energy_travelling} for a travelling-wave modulation is purely a consequence of the convergent characteristics. Here, we illustrate the same behaviour for the standing-wave modulation \eqref{eq:standing}. In (a), we plot the solution $u(x,t)$ when initiated by a pulse. As predicted by Proposition \ref{prop:energy}, the energy density, plotted in (b) follows the characteristics (dashed) and concentrates to the limit cycles.}\label{fig:energy}
	\end{figure}
	In \Cref{fig:energy}, we plot the solution $u(x,t)$, computed using a finite-difference discretisation, with initial conditions given by a gaussian pulse. As expected, the energy concentrates to the stable limit cycles of the underlying modulation. Here, we consider the standing-wave modulation of \Cref{fig:standing}.

	\subsection{Perturbation theory and asymptotic expansion of limit cycles}
	We now develop asymptotic expansions of limit cycles in the case of an index function $f$ which is perturbed by some magnitude  $\epsilon$.
	\begin{prop}\label{prop:asymptotic2}
		Assume $x_0(t)$ is a $(p,q)$-periodic limit cycle of an analytic, $(L,T)$-periodic function $f_0(x,t)$ such that 
		\begin{equation}\label{eq:assumption}
			r_0 = -\frac{1}{qT}\int_0^{qT} (f_0)_x(x_0(t),t)\mathrm{d} t \neq 0.\end{equation}
		Given an analytic, $(L,T)$-periodic function $f_1$ and for $0<\epsilon\ll1$, let $f_\epsilon(x,t) = f_0(x,t) + \epsilon f_1(x,t)$. Then, for small enough $\epsilon$, there is a $(p,q)$-periodic limit cycle $x_\epsilon(t)$ of $f_\epsilon(x,t)$.
	\end{prop}
	\begin{proof}
		We consider the problem
		\begin{equation}
			\begin{cases}\displaystyle \frac{\mathrm{d} x_\epsilon}{\mathrm{d} t} = f_\epsilon(x_\epsilon,t), \\ x_\epsilon(0) = a_\epsilon.
			\end{cases}
		\end{equation}
		We will choose some initial condition $a_\epsilon$, depending analytically on $\epsilon$, so that corresponding solution $x_\epsilon(t)$ will be a limit cycle. We let $y_\epsilon(t) = x_\epsilon(t) - t\frac{pL}{qT}$. Following similar arguments as in the proof of \Cref{prop:asymptotic}, $y_\epsilon$ can be written 
		\begin{equation}\label{eq:y}
			y_\epsilon(t) = \sum_{n=0}^\infty \epsilon^n y_n(t).
		\end{equation}
		The leading coefficient $y_0(t)$ satisfies the limiting problem
		\begin{equation}
			\begin{cases}\displaystyle \frac{\mathrm{d} y_0}{\mathrm{d} t} = f_0(x_0(t),t) - \frac{pL}{qT}, \\ y_0(0) = a_0,
			\end{cases}
		\end{equation}
		and we choose $a_0$ so that the limiting problem has a $qT$-periodic solution $y_0(t)$. For $n>0$, the coefficient $y_n(t)$ satisfies the linear differential equation
		\begin{equation}\label{eq:linearised}
			\begin{cases}\displaystyle \frac{\mathrm{d} y_n}{\mathrm{d} t} = (f_0)_x(x_0(t),t)y_n + g_n(t), \\ y_n(0) = a_n,
			\end{cases}
		\end{equation}
		for some function $g_n$. Here, $g_n(t)$ is defined solely in terms of $y_m(t)$ for $m<n$, and is a $qT$-periodic function whenever the functions $y_m(t)$ are. Moreover, $g_n(t)$ define a power series 
		\begin{equation}
			g_\epsilon(t) = \sum_{n=0}^\infty g_n(t)\epsilon^n
		\end{equation}
		with nonzero radius of convergence. We define the antiderivative $F(t)$ of $(f_0)_x(x_0(t),t)$:
		\begin{equation}
			F(t) = \int_0^t (f_0)_x(x_0(\tau),\tau) \mathrm{d}\tau.
		\end{equation}
		The solution to \eqref{eq:linearised} is then given by
		\begin{equation}
			y_n(t) = a_n\exp(F(t)) + \exp(F(t))\int_0^t g_n(\tau)\exp(-F(\tau)) \mathrm{d} \tau.
		\end{equation}
		Note that this is a $qT$-periodic function when $a_n$ is given by 
		\begin{equation}\label{eq:a}
			a_n = \frac{\exp(F(qT))}{1-\exp(F(qT))}\int_0^{qT} g_n(\tau)\exp(-F(\tau)) \mathrm{d} \tau.
		\end{equation}
		Note that $a_n$ satisfies
		\begin{equation}
			|a_n| < C\int_0^{qT} |g_n(\tau)| \mathrm{d} \tau
		\end{equation}
		for some constant $C$ independent of $n$. It follows that the coefficients $a_n$, defined inductively through \eqref{eq:a}, define an analytic function $a_\epsilon$:
		\begin{equation}
			a_\epsilon = \sum_{n=0}^\infty a_n \epsilon^n,
		\end{equation}
		for $\epsilon$ within some nonzero radius of convergence. This way, \eqref{eq:y} defines a solution $y_\epsilon(t)$ such that each order $y_n(t)$ is $qT$-periodic. It follows that $y_\epsilon(t)$ is also $qT$-periodic. 
	\end{proof}
	The key feature of the proof of \Cref{prop:asymptotic2} is that each order of perturbation admits a solution $y_n(t)$ with the same periodicity as the leading order limit cycle. We also emphasise that the main assumption \eqref{eq:assumption} does not depend on the perturbation $f_1$. If this condition holds, the limit cycle is robust and will persist under any small perturbation which respects the periodicity of $f_0$. Moreover, for each order of perturbation, the condition for stability with respect to initial conditions agrees with that of \eqref{eq:assumption}.

	\begin{figure}
		\centering
		\includegraphics[width=0.5\linewidth]{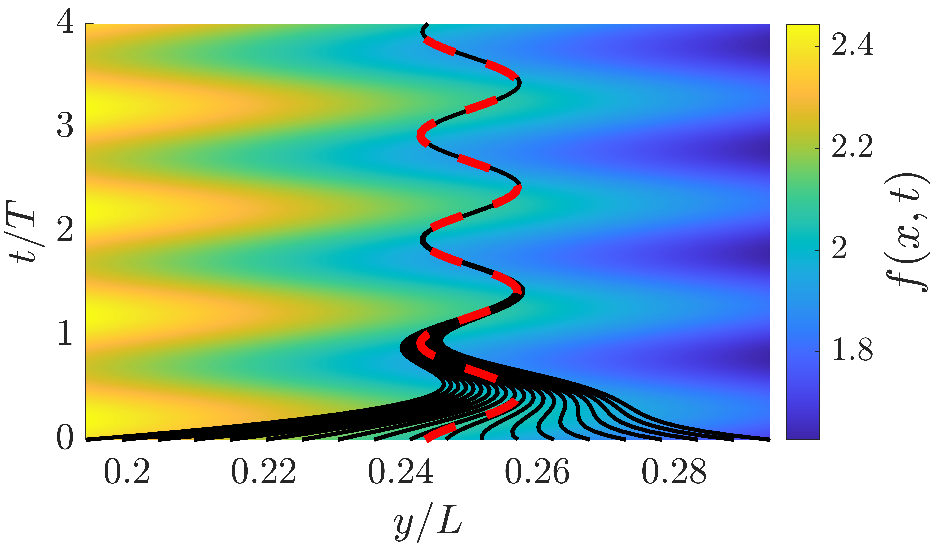}
		\caption{As long as $r\neq 0$, the limit cycles are robust against perturbations. Here, we compare numerically computed characteristics against the first-order asymptotic approximation of the perturbed limit cycle given by \eqref{eq:y1}.}\label{fig:pert}
	\end{figure}
	
	Next, we exemplify the perturbation theory by computing the solution in the case of a perturbation of a travelling wave:
	\begin{equation}
		f(x,t) = f_0(x-vt) + \epsilon f_1(x,t),
	\end{equation}
	for some $0<\epsilon\ll 1$. Transforming to a moving coordinate frame $y = x-vt$ gives
	\begin{equation}
		\frac{\mathrm{d} y}{\mathrm{d}t} = f_0(y) - v + \epsilon f_1(y+vt,t). 
	\end{equation}
	We consider perturbation theory analogously to \Cref{prop:asymptotic} up to order $\epsilon$, which gives 
	\begin{equation}
		\frac{\mathrm{d} y_0}{\mathrm{d}t} = f_0(y_0) - v,
	\end{equation}
	and 
	\begin{equation}
		\frac{\mathrm{d} y_1}{\mathrm{d}t} = y_1(t)f_0'(y_0) + f_1(y_0+vt,t). 
	\end{equation}
	The leading-order limit cycle is then given by a constant $y_0(t) = y_0$ such that $f_0(y_0) = v$. The next order is given by
	\begin{equation}\label{eq:y1}
		y_1(t) = a_1\exp(f_0'(y_0)t) + \exp(f_0'(y_0)t)\int_0^t f_1(y_0+v\tau,\tau)\exp(-f_0'(y_0)\tau)\mathrm{d} \tau.
	\end{equation}
	The condition \eqref{eq:assumption} is now given by $f_0'(y_0)\neq 0$. In this case, regardless of the choice of $f_1$, we can find a constant $a_1$ which gives a $T$-periodic function $y_1$ in \eqref{eq:y1}:
	\begin{equation}
		a_1 = \frac{\exp(f_0'(y_0)T)}{1-\exp(f_0'(y_0)T)}\int_0^T f_1(y_0+vt,t)\exp\bigl(-f_0'(y_0)t\bigr)\mathrm{d} t.
	\end{equation}
	The solution $y_1$ is illustrated in a concrete example in \Cref{fig:pert}. Here, we consider the double-travelling wave modulation as a perturbation of a single travelling wave: $f_0(x,t) =  2 + \cos\bigl(\frac{2\pi}{L}(x-vt)\bigr)$ while $f_1(x,t) = \cos\bigl(\frac{2\pi}{L}(x-2vt)\bigr)$.
	
	\subsection{Stability to non-periodic perturbations} \label{sec:nonperiodic}
	To further develop the perturbation theory of attractive limit cycles, we now consider stability with respect to non-periodic perturbations.	We assume $x_0(t)$ is a $(p,q)$-periodic limit cycle of the unperturbed problem
	\begin{equation}
		\begin{cases}\displaystyle \frac{\mathrm{d} x_0}{\mathrm{d} t} = f_0(x_0,t), \\ x_0(0) = a_0,
		\end{cases}
	\end{equation}
	and assume that  		
	\begin{equation}
		r_0 = -\frac{1}{qT}\int_0^{qT} (f_0)_x(x_0(t),t)\mathrm{d} t >0.
	\end{equation}
	We then consider a perturbed index function $f_\epsilon(x,t) = f_0(x,t) + \epsilon f_1(x,t)$ for some general, non-periodic function $f_1$ and consider the problem
	\begin{equation}
		\begin{cases}\displaystyle \frac{\mathrm{d} x_\epsilon}{\mathrm{d} t} = f_\epsilon(x_\epsilon,t), \\ x_\epsilon(0) = a_0.
		\end{cases}
	\end{equation}
	Since $f_1$ is non-periodic, we no longer expect a periodic limit cycle $x_\epsilon$, but we instead consider the asymptotic rate of attraction defined in \eqref{eq:rt}:
	\begin{equation} \label{eq:r_eps}
		r_\epsilon = -\lim_{t\to \infty}\frac{1}{t}\int_0^t (f_\epsilon)_x(x_\epsilon(\tau),\tau)\mathrm{d}\tau.
	\end{equation}
	If $r_\epsilon>0$, nearby characteristics will approach $x_\epsilon$ at exponential rate $r_\epsilon$. What we need to prove is that given $x_0$ is an attractive limit cycle, then $r_\epsilon>0$ for small enough $\epsilon$.
	
	Assume $(f_0)_x$ is uniformly Lipschitz and $(f_1)_x$ is bounded: $\bigl|(f_0)(x_1,t)-(f_0)(x_2,t)\bigr|\leq L|x_1-x_2|$ and $\bigl|(f_1)_x(x,t) \bigr| \leq M$ for all $x$ and $t$. Then 
	\begin{align}
		|r_\epsilon-r_0| &\leq \lim_{t\to \infty}\frac{1}{t}\int_0^t \bigl|(f_0)_x(x_\epsilon(\tau),\tau)-(f_0)_x(x_0(\tau),\tau)\bigr|\mathrm{d}\tau + \epsilon \lim_{t\to \infty}\frac{1}{t}\int_0^t \bigl|(f_1)_x(x_\epsilon(\tau),\tau)\bigr|\mathrm{d}\tau\\
		&\leq \lim_{t\to \infty}\frac{1}{t}\int_0^t L \bigl|x_\epsilon(\tau)-x_0(\tau)\bigr|\mathrm{d}\tau + \epsilon M.\label{eq:rdiff}
	\end{align}	
	Using the stability criterion $r_0>0$, we can follow the proof of \Cref{prop:asymptotic2} to show that as $\epsilon\to 0$ we have $\bigl|x_\epsilon(\tau)-x_0(\tau)\bigr|\leq \epsilon C$ for some constant $C$ independent of $t$.	Now, from \eqref{eq:rdiff} it follows that 
	\begin{equation}
		|r_\epsilon-r_0| \leq \epsilon CL + \epsilon M,
	\end{equation}
	so that, for small enough $\epsilon$ we have $r_\infty>0$, which shows that the characteristics remain attractive even in the case of a non-periodic perturbation.

    \subsection{Stability to random perturbations}

    Finally, to test the limits of the insight gained from our perturbative analyses, we can briefly consider adding some random noise to the modulation function $f$. For example, suppose the characteristic ODE $\mathrm{d}x/\mathrm{d}t=f(x,t)$ in \eqref{eq:char} is replaced by the SDE
	\begin{equation}
		\mathrm{d}X_t=f(X_t,t)\mathrm{d}t+\sigma \mathrm{d}W_t
	\end{equation}
	where $W_t$ is a standard Brownian motion and $\sigma>0$. This term describes stochastic perturbations to the characteristics due to random, spatially homogeneous kicks in the position variable. Let $p(x,t)$ be the probability distribution of $X_t$. This satisfies the Fokker–Planck equation \cite{pavliotis2014stochastic}
	\begin{equation}\label{eq:FP}
		\frac{\partial p}{\partial t} =-\frac{\partial }{\partial x} (f(x,t)p)+\frac{\sigma^2}{2}\frac{\partial^2 p}{\partial x^2 }.
	\end{equation}

    \begin{figure}
		\begin{subfigure}[t]{0.4\linewidth}
			\centering
			\includegraphics[height=3.75cm]{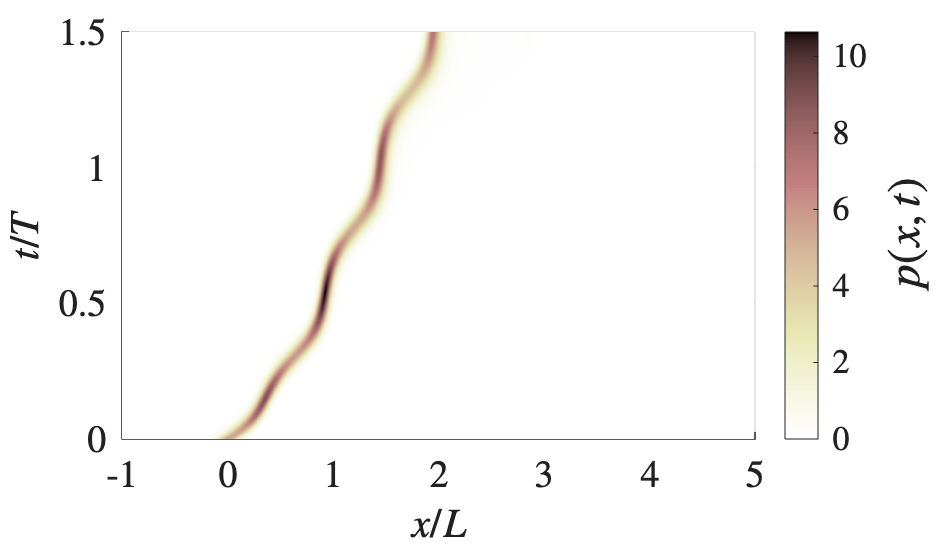}
			\caption{$\sigma^2=0.01$}
		\end{subfigure}\hfill
		\begin{subfigure}[t]{0.4\linewidth}
			\centering
			\includegraphics[height=3.75cm]{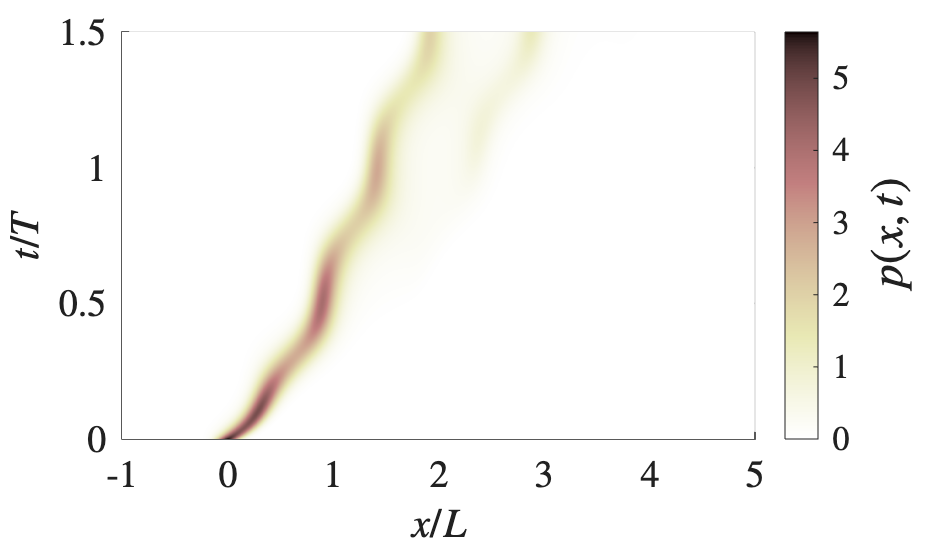}
			\caption{$\sigma^2=0.1$}
		\end{subfigure}
        \hfill
		\begin{subfigure}[t]{0.14\linewidth}
			\centering
			\includegraphics[height=3.75cm,trim={0 -0.15cm 0 0},clip]{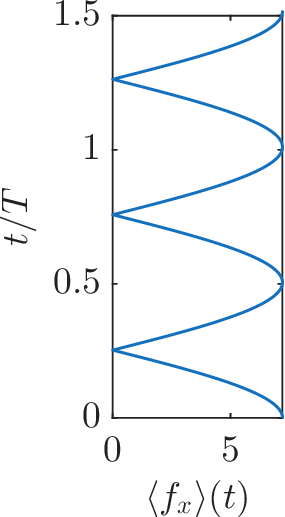}
			\caption{$\langle f_x \rangle(t)$} \label{fig:fx1}
		\end{subfigure}

\vspace{0.5cm}

        \begin{subfigure}[t]{0.4\linewidth}
			\centering
			\includegraphics[height=3.75cm]{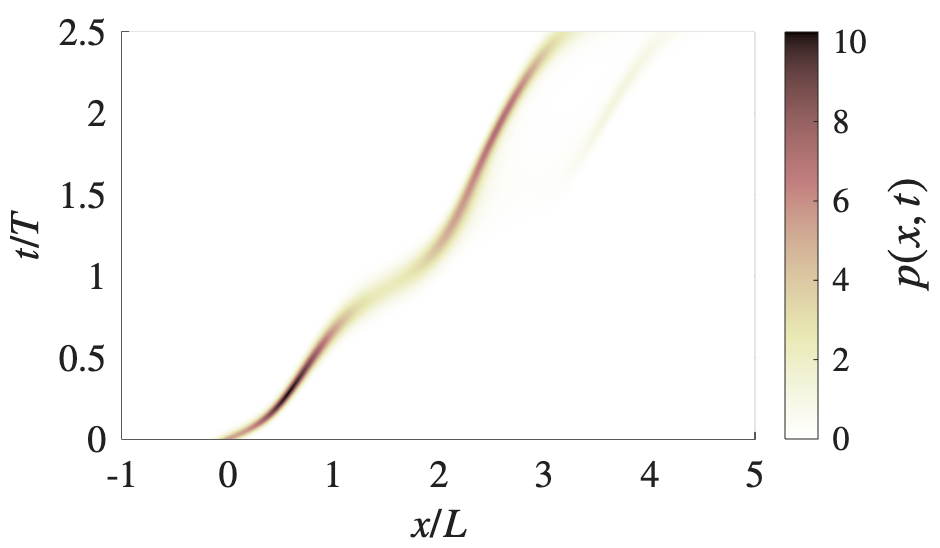}
			\caption{$\sigma^2=0.01$}
		\end{subfigure}\hfill
		\begin{subfigure}[t]{0.4\linewidth}
			\centering
			\includegraphics[height=3.75cm]{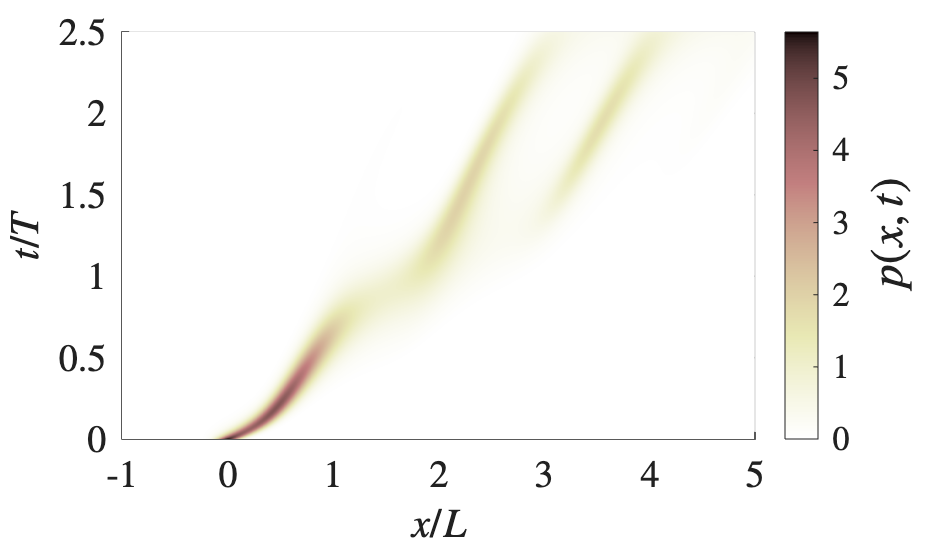}
			\caption{$\sigma^2=0.1$}
		\end{subfigure}
        \hfill
		\begin{subfigure}[t]{0.14\linewidth}
			\centering
			\includegraphics[height=3.75cm,trim={0 -0.15cm 0 0},clip]{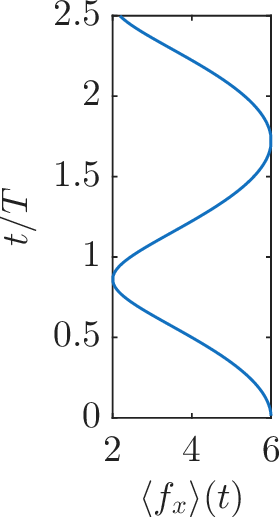}
			\caption{$\langle f_x \rangle(t)$} \label{fig:fx2}
		\end{subfigure}
		\caption{Solving Fokker-Planck for constant diffusion $\sigma$. The top row (a-c) corresponds to the standing wave modulation \eqref{eq:standing} and the bottom row (d-f) corresponds to the two travelling wave modulations \eqref{eq:doublewave}. In both cases, we show solutions $p(x,t)$ for two different values of the variance $\sigma^2$, as well as the evolution of the spatial average of the magnitude of the derivative $f_x$.}\label{fig:FokkerP}
	\end{figure}

    We can solve the Fokker-Planck equation \eqref{eq:FP} to understand how characteristic curves are likely to behave when randomness is added to the model. In \Cref{fig:FokkerP} we show solutions for two different values of the variance $\sigma^2$, which controls the diffusion in \eqref{eq:FP}. In both cases, the initial distributions are taken as very sharp Gaussians centred at zero (in particular, $p(x,0)=10/\sqrt{\pi} \exp(-100 x^2)$). We can see that the larger $\sigma$ is, the more likely it is a characteristic manages to escape to a nearby limit cycle. 

    The behaviour in \Cref{fig:FokkerP} is governed by the competition between the underlying attractive limit cycles and the diffusive term in \eqref{eq:FP}. From our previous perturbative analysis, we expect that the key quantity that governs the rate of attraction to the limit cycle is $f_x$. In \Cref{fig:fx1} and \Cref{fig:fx2} we show the evolution of the spatial average of the magnitude of $f_x$:
    \begin{equation}
        \langle f_x \rangle(t) := \frac{1}{L} \int_0^L |f_x(x,t)| \mathrm{d}x.
    \end{equation}
    This is a crude, back-of-the-envelope estimate for the global rate of attraction at the time $t$. Nevertheless, we see that this is able to predict the features of the solutions $p(x,t)$, such as how the extent to which $p(x,t)$ is concentrated along the limit cycle fluctuates periodically. The moment when $p(x,t)$ is least strongly confined near to the limit cycle, meaning characteristics are most likely to escape, are precisely the local minima of $\langle f_x \rangle(t)$.

	\section{Floquet–Fourier computational method}\label{sec:floquet}
	We next introduce the Floquet–Fourier numerical method used for computing limit cycles in \Cref{sec:main}. Related methods, using Fourier-series expansion of solutions to periodic differential equations, are well-established \cite{deconinck2006computing}, and we phrase it here for limit cycles of the space–time-periodic characteristic equation. The key step is to define the variable $z(t)$ as 
	\begin{equation}
		z(t) = \exp{\frac{2\pi \i x(t)}{L}}.
	\end{equation}
    We let $g(z,t)$ be the index function $f(x,t)$ as a function of $z$ instead of $x$:
    \begin{equation}
        g(z,t) = f\bigl(\tfrac{L}{2\pi\i}\log(z),t\bigr),
    \end{equation}
    for $z$ on the complex unit circle, which is independent of the branch of $\log$ due to the periodicity of $f(x,t)$. From \eqref{eq:char}, we find that $z$ solves the problem
	\begin{equation}
		\begin{cases}
			\displaystyle \frac{\mathrm{d} z}{\mathrm{d}t} = \frac{2\pi\i}{L}g(z,t)z(t), \\
			z(0) = z_0,
		\end{cases}
	\end{equation}
	where $z_0 = \exp{\frac{2\pi \i a}{L}}$. Note that $z(t)$ now satisfies $z(t+qT) = z(t)$, in other words, $z$ is a $qT$-periodic function. We may then expand $z(t)$ in a Fourier series
	\begin{equation}
		z(t) = \sum_{n=-\infty}^\infty z_n\exp\frac{2\pi\i n t}{qT}.
	\end{equation}
	Similarly, we may expand the $(L,T)$-periodic function $f(x,t)$ in a two-dimensional Fourier series as
	\begin{equation}
		f(x,t) = \sum_{k,\ell\in \Z} f_{k\ell} \exp{\frac{2\pi \i k x}{L}} \exp{\frac{2\pi \i \ell t}{T}}.
	\end{equation}
	Substituting the expression for $z$, we then have 
	\begin{equation}
		g(z,t) = \sum_{k,\ell\in \Z} f_{k\ell} z^k(t) \exp{\frac{2\pi \i \ell t}{L}}.
	\end{equation}
	Here, $z^k(t)$ may be written
	\begin{equation}
		z^k(t) = \sum_{n=-\infty}^\infty (z^{k,*})_n\exp\frac{2\pi\i n t}{qT},
	\end{equation}
	where we use the notation $(z^{k,*})_m$ for the $k$-fold convolution of $z_n$ by itself. In total, we have
	\begin{equation}
		\sum_{n=-\infty}^\infty \frac{2\pi\i n}{qT} z_n\exp\frac{2\pi\i n t}{qT} = \frac{2\pi \i}{L}\sum_{k,\ell,n\in \Z} f_{k\ell}(z^{k+1,*})_\ell\exp\frac{2\pi\i (n + q\ell) t}{qT}.
	\end{equation}
	Simplifying, we obtain the nonlinear system of equations 
	\begin{equation}\label{eq:floquet_cycle}
		z_n = \frac{q}{vn}\sum_{k,\ell\in \Z}f_{k\ell} (z^{k+1,*})_{n-q\ell},
	\end{equation}
	for every $n \in \Z\setminus\{0\}$. A limit cycle is then specified by an integer $q\in\Z$ for which there is a nontrivial sequence $(z_n)$ satisfying \eqref{eq:floquet_cycle}. We can directly predict the rate $r$ of attraction of the limit cycles from \eqref{eq:stability}, which gives
	\begin{equation}
		r = - \frac{2\pi\i}{L}\sum_{k,\ell\in \Z}kf_{k\ell}  (z^{k,*})_{-\ell}.
	\end{equation}
	
	If $f(x,t)$ is smooth, the coefficients $f_{mn}$ will decay rapidly and we truncate the system in \eqref{eq:floquet_cycle} to finite $-M\leq \ell\leq M$. We also note that, since $(z^{-m,*})_n$ are the Fourier coefficients of $z^{-m}(t)$, it follows that 
	\begin{equation}
		(z^{-m,*})_n = \overline{(z^{m,*})_{-n}},
	\end{equation}
	so convolutions of negative order are not required for numerical implementation of \eqref{eq:floquet_cycle} and, after truncating the series, \eqref{eq:floquet_cycle} is a system of polynomial equations in $z_n$ and $\overline{z_n}$. This system can be solved using iterative Newton-like methods, and the \texttt{MATLAB}-function \texttt{fsolve} was used to find the limit cycles in \Cref{sec:main}. As initial guess for the iterative solution method, we use a straight line $x(t) = x_0 + vt$.
	
	Since the Floquet–Fourier method does not rely on a direct integration of the ODE, it is possible to compute both stable and unstable limit cycles. By contrast, a forward integration of the ODE will only yield the stable limit cycles. 
	
	\section{Full-wave simulations}
	We now return to the full wave equation \eqref{eq:wave} and study the energy accumulation using finite-difference discretisations. The energy is computed by a symmetric finite-difference truncation of \eqref{eq:E}. This way, we are able to compare the behaviour of impedance-matched systems (studied in previous sections) to general systems.
	
	In \Cref{fig:fullwave}, we consider a travelling-wave modulation, and compare an impedance-matched system with a non-magnetic system where $\mu(x,t)$ is held constant and only $\epsilon(x,t)$ is modulated. Here, we take $f(x,t) = 2+\cos\left(\frac{2\pi}{L}(x-vt)\right)$ and $\epsilon(x,t) = c_0/f(x,t)$. In the impedance-matched case, we take $\mu(x,t) = \epsilon(x,t)$, while in the non-magnetic case we take $\mu(x,t) = c_0/2$. As initial conditions, we take a rightward-propagating Gaussian pulse $u(x,0) = \exp(-x^2/0.4L)$.  In both cases, the resulting wave field is mainly propagating to the right. Although the non-magnetic case is not covered by the analysis of \Cref{sec:main}, we observe energy accumulation in a similar way to the impedance-matched case. Here, the main difference is a lower rate of energy accumulation in the non-magnetic system.
	
	\begin{figure}[tbh]
		\begin{subfigure}[t]{0.48\linewidth}
			\centering
			\includegraphics[width=0.9\linewidth]{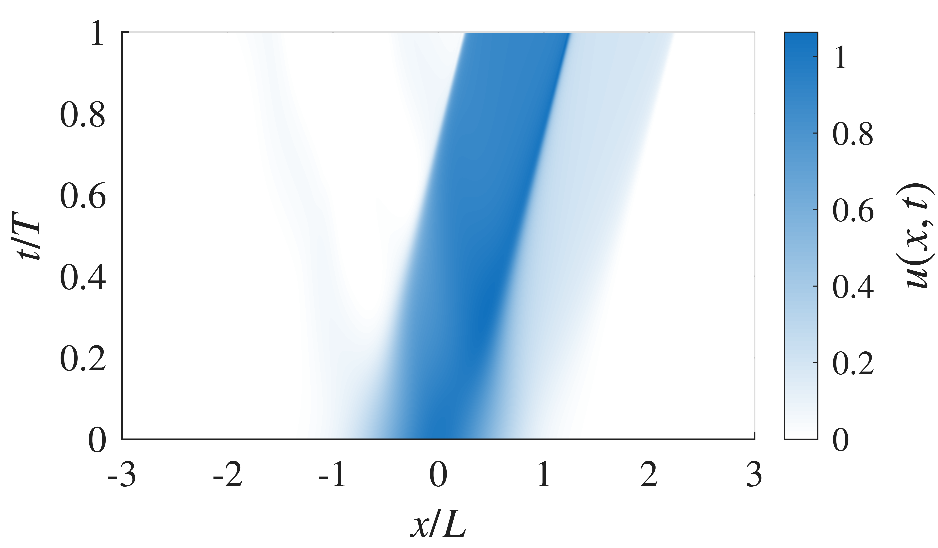}
			\caption{Solution with impedance-matched modulation.}
		\end{subfigure}\hfill
		\begin{subfigure}[t]{0.48\linewidth}
			\centering
			\includegraphics[width=0.9\linewidth]{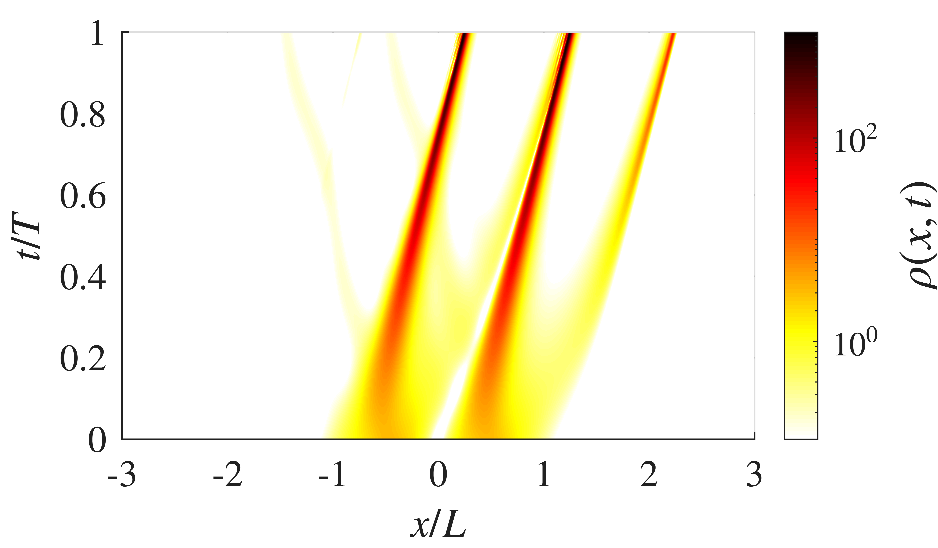}
			\caption{Energy density with impedance-matched modulation.}
		\end{subfigure}
		\begin{subfigure}[t]{0.48\linewidth}
			\centering
			\includegraphics[width=0.9\linewidth]{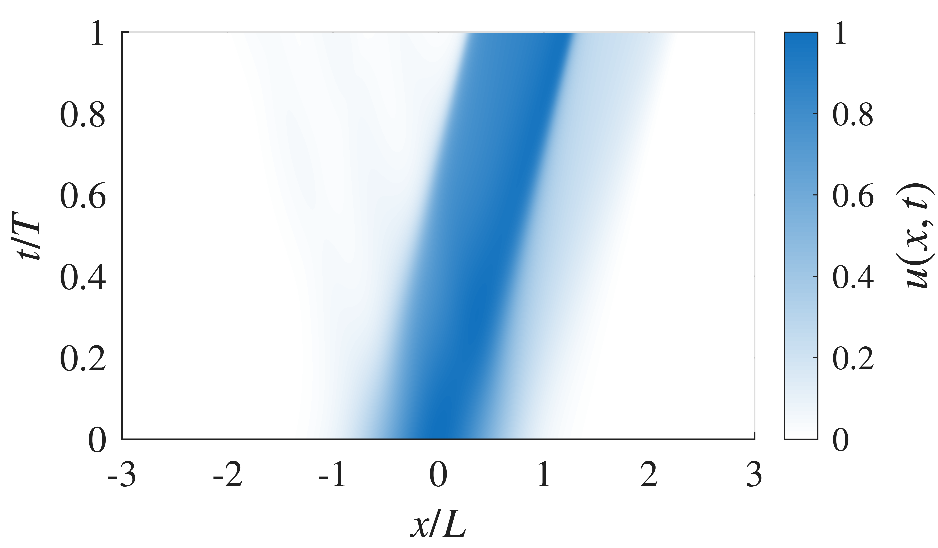}
			\caption{Solution with non-magnetic modulation.}
		\end{subfigure}\hfill
		\begin{subfigure}[t]{0.48\linewidth}
			\centering
			\includegraphics[width=0.9\linewidth]{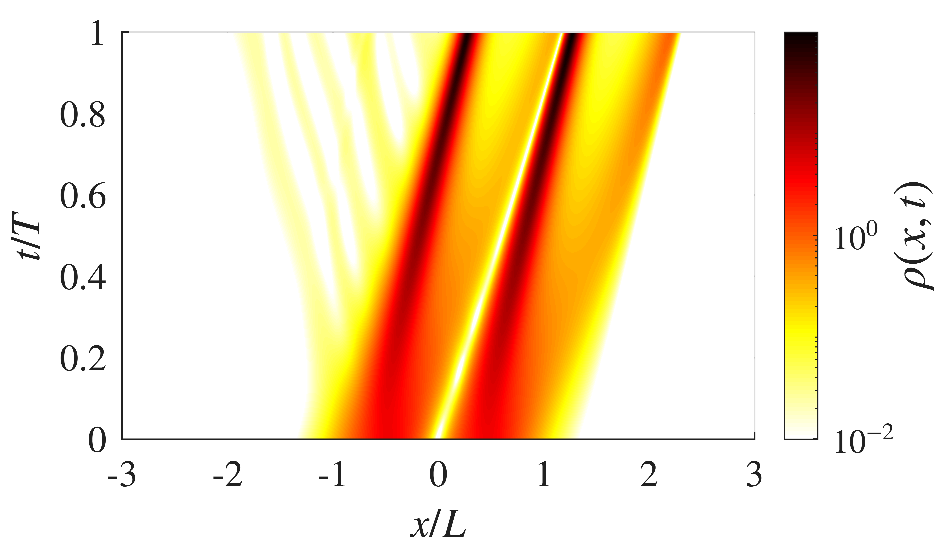}
			\caption{Energy density with non-magnetic modulation.}
		\end{subfigure}
		\caption{We consider the wave field $u$ and the energy density $\rho$ of the full wave equation \eqref{eq:wave} when initialized by a rightward-propagating Gaussian pulse. In panels (a) and (b), we consider an impedance-matched system, and observe energy accumulation along the limit cycles as predicted by \Cref{prop:energy}. In panels (c) and (d), we consider a non-magnetic system where $\epsilon(x,t) = c_0/f(x,t)$ and $\mu(x,t)= c_0/f_0$ for constant $f_0$. This system shows similar features to (a)--(b), although the rate of energy accumulation is decreased.}\label{fig:fullwave}
	\end{figure}
	
    \section{Concluding remarks}
    We have shown that wave-energy accumulation in smooth space–time-varying media is governed by stable limit cycles. Neighbouring characteristics converge exponentially, causing the wave energy to concentrate and grow at a rate determined by the corresponding Lyapunov exponent. We developed analytical estimates for this rate, established robustness under small perturbations, and introduced a Floquet–Fourier method for computing both stable and unstable cycles. Full-wave simulations confirm that the same mechanism remains relevant beyond the impedance-matched setting.

    These results provide a common dynamical framework that unifies the widely studied modulation patterns, such as travelling-wave \cite{horsley2023quantum}, standing-wave \cite{horsley2024traveling} and, with a few modifications, rectangular \cite{lurie2006wave} modulation. It allows the key energy-accumulation phenomena to be generalised to more exotic periodic modulation patterns, such as the double-travelling-wave modulation considered as a demonstrative example here. Our analysis shows that the existence of attractors for characteristics is the key underlying mechanism for energy accumulation. This suggests natural extensions to almost-periodic \cite{koufidis2024enhanced}, quasiperiodic \cite{ruchiev2026energy} and randomly \cite{carminati2021universal} modulated media, where limit cycles may be replaced by more general attracting trajectories or invariant sets while energy accumulation persists. Under suitable assumptions, we expect that some of these examples could be handled using the non-periodic perturbation theory developed in \Cref{sec:nonperiodic}.
    
    \section{Acknowledgements}

    BD was funded by the Leverhulme Trust through a Research Leadership Award.
	
	\section{Competing interests}
	
	The authors have no competing interests to declare.
	
	\section{Data availability}
	
	The software developed for this study is available at \url{https://doi.org/10.5281/zenodo.23317726}.
	
	\bibliographystyle{abbrv}
	\bibliography{references}{}
\end{document}